\documentclass[10pt]{article}
\usepackage[utf8]{inputenc}
\usepackage[a4paper,top=2.5cm,bottom=2cm,left=2cm,right=2cm,marginparwidth=1.5cm]{geometry}
\usepackage{amsmath}
\usepackage{amssymb}
\numberwithin{equation}{section}
\usepackage[english]{babel}
\usepackage{amsthm}
\usepackage{bbm}
\usepackage{amsfonts}
\usepackage{subcaption}
\usepackage{comment}
\usepackage{mathrsfs}
\usepackage{hyperref}
\usepackage{titlesec}
\usepackage{bbm}
\usepackage[numbers,sort]{natbib}
\hypersetup{
    colorlinks=true,
    linkcolor=blue,
    citecolor=red,
    filecolor=magenta,      
    urlcolor=black,
    pdftitle={3DWeak},
    pdfpagemode=FullScreen
    }
\usepackage{graphicx, tikz}
\usepackage{float}
\usepackage{xcolor}
\usepackage{bbm}
\usepackage[format=plain,
            font=it]{caption}

\newtheorem{theorem}{Theorem}[section]
\newtheorem{corollary}{Corollary}[theorem]
\newtheorem{lemma}[theorem]{Lemma}
\newtheorem{proposition}[theorem]{Proposition}
\newtheorem*{remark}{Remark}

\newtheorem{definition}[theorem]{Definition}

\newcommand{\upd}{\mathrm{d}}

     \def\d{\delta}          \def\D{\Delta }
    \def\ve{\varepsilon}

    \def\m{\mu}             \def\n{\nu}         
    \def\r{\rho}                 
                   
   \def\f{\phi }                
                         
       \def\w{\omega }

\title{Non-singular hotspots between closely spaced high-index nanoparticles}

\author{Konstantinos Alexopoulos\thanks{Centre de Mathématiques Appliquées, École Polytechnique, 91120 Palaiseau, France.} \and Bryn Davies\thanks{Mathematics Institute, University of Warwick, Coventry CV4 7AL, UK.} \and Pierre Millien\thanks{Institut Langevin, ESPCI Paris, PSL University, CNRS, 1 Rue Jussieu, 75005 Paris, France.}}
\date{}

\begin{document}

\maketitle

\begin{abstract}
    We study the concentration of the field between two nearly touching high-index dielectric resonators in three dimensions. The model is a scalar Helmholtz transmission problem in the resonant regime, wherein the wavelength inside the resonators of the same order as their typical diameter. The material contrast enters only a lower-order term of the equation and not its principal part. As a consequence, the gradient of the field stays bounded independently of the distance separating the particles, and does not blow up. Numerically, however, the gradient still concentrates in the gap: as the particles approach, it grows like the inverse of their separation over an intermediate range of distances, and then saturates once they are very close. We explain this pre-asymptotic effect through a weak-coupling regime, in which the resonant modes of the pair are perturbations of the modes of each isolated particle. When such a mode keeps a nonzero contrast between the two facing boundaries, a mean value argument across the gap accounts for the growth of the gradient; this requires strengthening the standard spectral perturbation estimates from an average to a pointwise control. The growth stops once the interaction between the particles is no longer weak. Numerical experiments confirm the transition from amplification to saturation.
\end{abstract}

\section{Introduction}

\subsection{Context}
When two high-index particles are brought close to contact, a strong concentration of the electromagnetic field appears in the narrow gap between them. In coupled dielectric nanostructures this concentration can take the form of an electric or magnetic hotspot: a localized peak of the  near-field at the point of closest approach. This effect has been documented numerically and experimentally \cite{bakker2015magnetic, mirzaei2015electric, calandrini2018magnetic, pascale2019full, van2016direct}. These works report the enhancement and its dependence on particle proximity,  but do not identify the mathematical mechanism that explains how the gradient in the gap scales with the separation distance.

For elliptic equations with discontinuous coefficients, the formation of a hotspot for the gradient of the solution when two heterogeneities are close to each other has been studied in several settings.  In the general case of an uniformly elliptic system with finite coefficients, the gradient remains bounded independently of the separation distance between the different inclusions  \cite{bonnetier2000elliptic,li2000gradient,li2003estimates} and there is no blow up.  However,  if the contrast between the coefficients in the inclusion and the background go to zero or infinity the gradient may explode as the inclusions are asymptotically close. For the
conductivity equation $\nabla\cdot(a\nabla u)=0$ with $a$ piecewise constant,
equal to a value $k$ inside the inclusions and to $1$ in the background, the
gradient between two inclusions blows up as they approach contact when the
contrast $k$ tends to infinity (perfect conductor) or to zero (insulating
inclusion or void); \cite{ammari2005gradient,
ammari2007estimates, ammari2015elliptic, bao2009gradient, ciraolo2019gradient,
chen2021optimal, kang2021singular},  while of course for a finite $k$ it stays bounded. The same mechanism produces stress concentration in the Lam\'e system of linear elasticity, with stiff inclusions
or holes \cite{kang2019quantitative, ammari2013spectral, lim2017stress}. 
Inspired by physical observations, there has been a major effort in the mathematics community to extend these type of \emph{gradient blow up} results to larger class of dimers of subwavelength resonators : bubbles in fluids,  plasmonic nano-particles,  di-electric nanoparticles\ldots  \cite{ammari2020close,yu2019hybridization,bonnetier2019plasmonic, li2025resonant}. We refer to the recent review \cite{kang2021singular} for a more comprehensive list of the different cases.
The framework for this article is the study of the field concentration between high index dielectric resonators.  The correct governing model is the Maxwell system,  in which the material contrast is carried by the electric permittivity. In a
two-dimensional, $z$-invariant geometry it decouples into two scalar problems,
one for each field aligned with the invariance axis. The electric component
$E_z$ satisfies
\begin{align}\label{eq:scalarE}
    \Delta E_z + \omega^2 \varepsilon\, E_z = 0,
\end{align}
where the contrast enters the zeroth-order term, while the magnetic component
$H_z$ satisfies
\begin{align*}
    \nabla\cdot\big(\varepsilon^{-1}\nabla H_z\big) + \omega^2 H_z = 0,
\end{align*}
where the contrast enters the principal, divergence, term \cite{moiola2019acoustic}.  The field concentration problem has been studied for the second case \cite{deng2022optimal} and for the Maxwell system in the asymptotic regime $\varepsilon\to \infty,  \varepsilon\omega \sim 1$ \cite{deng2022gradient}. 
However,  to the best of our knowledge,  no work has been published on the case $\varepsilon \to \infty ,  \varepsilon\omega^2 \sim 1$ which is the resonant scaling regime for high index nanoparticles \cite{meklachi2018asymptotic,ammari2023mathematical}. Inspired by these scalar reductions,  we aim at studying the gradient concentration problem for \eqref{eq:scalarE} in the \emph{high index resonant regime}, $\varepsilon\gg 1 ,  \varepsilon\omega^2 \sim 1$.  We deliberately adopt a point of view that is close to the applications : we fix the material parameters, frequency and size of the particles so that we are in the resonant regime and then we study the behavior of the gradient of the field when the spacing between the particles go to zero, while all the other parameters are fixed. 

\subsection{Position of the problem}
Of course, in our setting, with the physical parameters fixed and non-degenerated,  since the principal part of the operator is constant, an elementary analysis give that the solution is uniformly in $H^2$ and the gradient remains bounded in the $L^\infty$ norm regardless of the separation distance $\kappa$.  However, as can be seen on Figure \ref{fig:3d_radial_hotspot} there is a gradient concentration phenomenon when the particles are close to each other,  even if there is no asymptotic blow up when the separation distance $\kappa$ goes to zero.  Numerical experiments show that when the particles approach each other,  the gradient grows as $\kappa^{-1}$ and then saturates as $\kappa\to 0$, as shown in Figure \ref{fig:3D_all_modes_MVT}. The goal of this paper is to understand this \emph{pre-asymptotic} concentration effect and the mathematical mechanism behind it. 

\subsection{Main results and contributions}

The main result of the paper is Theorem \ref{thm:hotspot}.  It states that for some of the modes of the dimer system,  the gradient behaves as the inverse of the separation distance $\kappa$ over a range of distances.  Moreover,  this local  $\kappa^{-1}$ scaling can be understood \emph{via} an elementary \emph{mean value theorem} argument. 

The analysis rests on the spectral study of the coupled integral operator
\eqref{pb} associated with the transmission problem \eqref{Helmholtz problem 3D}.
Its central object is a \emph{weak-coupling regime}: the range of separations
over which the interaction between the two resonators is smaller than their
self-interaction. In this regime the coupled resonant modes are explicit
perturbations of the isolated single-particle modes. The perturbation is
naturally controlled in $L^2$, the setting in which the uncoupled operator is
self-adjoint; but the mean value argument needs the modes evaluated at the facing
boundary points, so we lift this $L^2$ smallness to a pointwise $L^\infty$
estimate (Section~\ref{sec:C0_control}). This pointwise control is the
technical core of the paper: it guarantees that the boundary contrast survives
the coupling and produces the $\kappa^{-1}$ gradient. 

Finally, we illustrate the phenomenon numerically, computing the modes of the
coupled three-dimensional system and reconstructing the field outside the
particles through the Lippmann--Schwinger representation.

\section{Mathematical framework}\label{sec:magnetic hotspot}

In this section, we introduce the mathematical framework used throughout the paper. We begin by describing the geometric configuration of two closely spaced dielectric resonators and formulate the associated scattering problem. We then derive a Lippmann--Schwinger representation of the solution and identify the quasi-static Newtonian operator that governs the leading-order resonant behavior in the subwavelength regime. This formulation provides the basis for the weak-coupling analysis developed in the following sections.

\subsection{Geometry}

Let $D_1$ and $D_2$ be two three-dimensional dielectric resonators. The resonators are composed of the same dielectric material, characterized by a frequency-dependent permittivity $\varepsilon_D$, and are embedded in a homogeneous background with constant permittivity $\varepsilon_0$ and permeability $\mu_0$. The distance between the boundaries of $D_1$ and $D_2$ is denoted by $\kappa>0$, and we are interested in the nearly touching regime $\kappa \to 0$. The characteristic size of the system is denoted by $\delta$, and we assume the subwavelength regime
\begin{align}
    \delta k_0 \ll 1,
    \qquad
    k_0 = \omega \sqrt{\varepsilon_0\mu_0}.
\end{align}
A schematic depiction of this geometry is shown in Figure~\ref{fig:spheres}, where the particles $D_1$ and $D_2$ are represented as spheres of radii $\rho_1$ and $\rho_2$, respectively.

\begin{figure}[ht]
    \centering
    \includegraphics[width=0.55\linewidth]{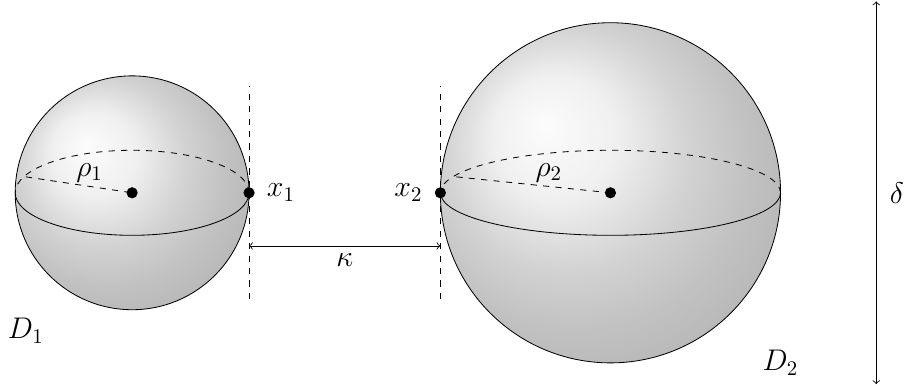}
    \caption{Geometry of the two-resonator configuration. The particles $D_1$ and $D_2$ are modeled as spheres of radii $\rho_1$ and $\rho_2$, respectively, embedded in a homogeneous background medium. Their boundary-to-boundary separation is denoted by $\kappa$, and the nearly touching regime corresponds to $\kappa\to0$.}
    \label{fig:spheres}
\end{figure}

\subsection{Governing equations}

We consider the scattering of time-harmonic scalar waves by two dielectric resonators embedded in a homogeneous background medium. The resonant response is described by a scalar field $u$, and the hotspot is measured through the concentration of the gradient $\nabla u$. The total field $u$ satisfies the Helmholtz transmission problem
\begin{align}\label{Helmholtz problem 3D}
    \begin{cases}
    \Delta u + \delta^2 \omega^2 \varepsilon_D\mu_0\, u = 0 & \text{in } D := D_1 \cup D_2, \\
    \Delta u + \delta^2 k_0^2 u = 0 & \text{in } \mathbb{R}^3 \setminus \overline{D},\\
    u|_+ - u|_-=0 &\text{ on } \partial D_i, \quad i=1,2, \\
    \frac{\partial u}{\partial \n}|_+ - \frac{\partial u}{\partial \n}|_- = 0 &\text{ on } \partial D_i, \quad i=1,2, \\
    u(x) - u_{in}(x) &\text{ satisfies the outgoing radiation condition as } |x|\to \infty,\\
    \end{cases}
\end{align}
where $u_{in}$ is the incident wave, assumed to satisfy
\begin{align*}
    (\D + k_0^2)u_{in}=0 \ \ \text{ in } \mathbb{R}^3,
\end{align*}
and the appropriate outgoing radiation condition is the Sommerfeld radiation condition, which requires that
\begin{equation}\label{Sommerfeld}
    \lim_{|x|\to\infty} |x|
    \left( \frac{\partial}{\partial |x|} - i k_0 \right)\Big[u(x) - u_{in}(x)\Big]=0.
\end{equation}

\subsection{Integral formulation}

Let $G^k(x)$ be the outgoing Helmholtz Green's function in $\mathbb{R}^3$, defined as the unique solution to
\begin{align*}
    (\D + k^2) G^k(x) = \d_0(x) \ \ \text{ in } \mathbb{R}^3,
\end{align*}
along with the outgoing radiation condition (\ref{Sommerfeld}). It is well known that $G$ is given by
\begin{align}
    G^k(x)= -\frac{e^{ik|x|}}{4\pi|x|}.
\end{align}
We have the following integral representation formula for the solution $u$ of \eqref{Helmholtz problem 3D}, as stated in \cite{AD}.

\begin{theorem}[Lippmann-Schwinger Integral Representation Formula] \label{Integral representation}
    The solution to the Helmholtz problem \eqref{Helmholtz problem 3D} is given by
    \begin{align}\label{eq:LS_3D}
        u(x)-u_{in}(x) = -\d^2 \w^2 \xi_D \int_D G^{\d k_0}(x-y) u(y) \upd y, \ \ x\in\mathbb{R}^3,
    \end{align} 
    where the function $\xi_D:\mathbb{C}\to\mathbb{C}$ describes the permittivity contrast between $D$ and the background and is given by
    \begin{align}
        \xi_D = \m_0\Big(\ve_D - \ve_0\Big).
    \end{align}
\end{theorem}

Based on \eqref{eq:LS_3D}, we define the operator $\mathcal{K}^{\d k_0}$ by
\begin{align}
    \begin{aligned}
        \mathcal{K}^{\d k_0}_D: \ \mathrm{L}^{2}&(D) \ &&\longrightarrow \qquad \qquad  \mathrm{L}^{2}(D) \\ 
        &u &&\longmapsto  \int_D G^{\d k_0}(x-y) u(y) \upd y. 
    \end{aligned}
\end{align}

In order to study the behavior of the system of nanoparticles, we will characterize the properties of the operator $K^{\d k_0}_D$. We observe that the Lippmann-Schwinger formulation \eqref{eq:LS_3D} of the problem is equivalent to
\begin{align*}
    (u-u_{in})(x) = \d^2 \w^2 \xi_D \mathcal{K}^{\d k_0}_D[u](x).
\end{align*}
This is equivalent to
\begin{align*}
    \Big(\text{Id} - \d^2 \w^2 \xi_D K^{\d k_0}_D\Big)[u](x) = u_{in}(x), 
\end{align*}
for all $x\in D$, where $\text{Id}$ denotes the identity operator. Then, the subwavelength resonance problem is to find $\w \in \mathbb{C}$ close to 0, such that the operator $(I - \d^2 \w^2 \xi_D \mathcal{K}^{\d k_0}_D)^{-1}$ is singular, or equivalently, such that there exists $u \in L^2(D)$, $u \ne 0$ with
\begin{align}\label{pb}
    u(x) - \d^2 \w^2 \xi_D \int_D G^{\d k_0}(x-y) u(y) \upd y = 0, \ \ \ \text{ for } x \in D.
\end{align}

\subsection{Quasi-static limit}

Let us define the Newtonian potential on $D$ to be $\mathcal{K}^{(0)}_D: L^2(D)\to L^2(D)$ such that
\begin{equation}\label{eq:K_3D_static}
    \mathcal{K}^{(0)}_D[u](x) := \int_D u(y) G^0(x-y) \upd y = - \frac{1}{4\pi} \int_D \frac{1}{|x-y|}u(y)\upd y.
\end{equation}
Similarly, we define the operators $\mathcal{K}^{(n)}_D: L^2(D) \to L^2(D)$, for $n=1,2,\dots,$ as
\begin{equation}\label{Kn}
    \mathcal{K}^{(n)}_D[u](x) = - \frac{i}{4\pi} \int_D \frac{(i|x-y|)^{n-1}}{n!}u(y)\upd y.
\end{equation}
Then, the operator $\mathcal{K}^{\d k_0}_D$ admits the following asymptotic expansion for small $\d$, as stated in \cite{AD}.
\begin{lemma}[\cite{AD}] \label{Taylor d=3}
    The operator $\mathcal{K}^{\d k_0}_D$ can be rewritten as
    \begin{align*}
        \mathcal{K}^{\d k_0}_D = \sum_{n=0}^{\infty} (\d k_0)^n \mathcal{K}^{(n)}_D,
    \end{align*}
    where the series converges in the $L^2(D) \to L^2(D)$ operator norm if $\d k_0$ is small enough.
\end{lemma}

Hence, in the subwavelength regime $\delta k_0\ll1$, the operator admits the expansion
\begin{align}\label{eq:K_3D_expansion}
    \mathcal K_D^{\delta k_0} = \mathcal K_D^{(0)} + O(\delta k_0),
\end{align}
in the $L^2(D)\to L^2(D)$ operator norm, where the leading-order static operator is the Newtonian volume potential $\mathcal K_D^{(0)}$. Hence, the dominant spectral structure is governed directly by the Newtonian operator $\mathcal K_D^{(0)}$.

The resonant frequencies are determined by the nonlinear eigenvalue problem associated with \eqref{eq:LS_3D}. In the subwavelength regime, the leading-order modal structure can therefore be studied through the spectral properties of $\mathcal K_D^{(0)}$ and its two-particle interaction structure.

\section{Weak coupling regime} \label{sec:weak coupling}

We now introduce the weak-coupling framework that forms the core of our analysis. The key idea is to separate the intrinsic resonant behavior of each isolated resonator from the interaction effects generated by their proximity. This is achieved through a block decomposition of the quasi-static interaction operator. We then use the spectral structure of the isolated resonators to define the weak-coupling regime and derive perturbative descriptions of the coupled resonant modes.

\subsection{Block decomposition}

Since the resonator domain consists of two particles, we use the Hilbert space decomposition
\begin{align*}
    L^2(D) = L^2(D_1)\oplus L^2(D_2).
\end{align*}
With respect to this decomposition, the static operator can be written in block form as
\begin{align}\label{eq:blockdecomp_3D}
    \mathcal K_D^{(0)} =
    \begin{pmatrix}
        \mathcal K_{11}^{(0)} & \mathcal K_{12}^{(0)}\\
        \mathcal K_{21}^{(0)} & \mathcal K_{22}^{(0)}
    \end{pmatrix},
\end{align}
where
\begin{align*}
    \mathcal K_{ij}^{(0)}:L^2(D_j)\to L^2(D_i)
\end{align*}
is given by
\begin{align}
    \mathcal K_{ij}^{(0)}[u](x) = \mathcal K_{D_j}^{(0)}[u](x), \qquad \text{with} \quad x\in D_i,
\end{align}
for $i,j=1,2$. The diagonal blocks $\mathcal K_{11}^{(0)}$ and $\mathcal K_{22}^{(0)}$ describe the self-interaction of each isolated resonator, while the off-diagonal blocks $\mathcal K_{12}^{(0)}$ and $\mathcal K_{21}^{(0)}$ encode the interaction between the two particles.

We therefore decompose
\begin{align}\label{eq:diagoff_3D}
    \mathcal K_D^{(0)} = \mathcal K_{\mathrm{diag}}^{(0)} + \mathcal K_{\mathrm{off}}^{(0)},
\end{align}
where
\begin{align}
    \mathcal K_{\mathrm{diag}}^{(0)} :=
    \begin{pmatrix}
        \mathcal K_{11}^{(0)} & 0\\
        0 & \mathcal K_{22}^{(0)}
    \end{pmatrix},
    \qquad
    \mathcal K_{\mathrm{off}}^{(0)} :=
    \begin{pmatrix}
        0 & \mathcal K_{12}^{(0)}\\
        \mathcal K_{21}^{(0)} & 0
    \end{pmatrix}.
\end{align}
This decomposition separates the intrinsic resonant behavior of the isolated particles from the interaction effects caused by their proximity.

\subsection{Isolated spectral modes}

The diagonal operator $\mathcal K_{\mathrm{diag}}^{(0)}$ is the direct sum of the two isolated single-particle operators. Thus, its spectrum is obtained from the spectra of $\mathcal K_{D_1}^{(0)}$ and $\mathcal K_{D_2}^{(0)}$.

Let us denote by 
\begin{align}
    \{\rho_n^{(1)},f_n^{(1)}\}_{n\geq1}
    \qquad\text{ and }\qquad
    \{\rho_n^{(2)},f_n^{(2)}\}_{n\geq1},
\end{align}
the eigenpairs of the operators $\mathcal{K}^{(0)}_{11}$ and $\mathcal{K}^{(0)}_{22}$, respectively. Then the eigenvectors of $\mathcal{K}^{(0)}_{\mathrm{diag}}$ are obtained by lifting the isolated eigenvectors to $L^2(D_1) \oplus L^2(D_2)$. Namely, we define
\begin{align}
    \Phi_j^{(1)} :=
    \begin{pmatrix}
        f_j^{(1)}\\
        0
    \end{pmatrix},
    \qquad \text{ and } \qquad
    \Phi_j^{(2)} :=
    \begin{pmatrix}
        0\\
        f_j^{(2)}
    \end{pmatrix}.
\end{align}
Then,
\begin{align}
    \mathcal K_{\mathrm{diag}}^{(0)}\Phi_j^{(1)} = \rho_j^{(1)}\Phi_j^{(1)}
    \qquad \text{ and } \qquad
    \mathcal K_{\mathrm{diag}}^{(0)}\Phi_j^{(2)} = \rho_j^{(2)}\Phi_j^{(2)}.
\end{align}

When the particles are separated by a positive gap $\kappa$, the off-diagonal operator $\mathcal K_{\mathrm{off}}^{(0)}$ couples these isolated modes. The coupled eigenmodes are therefore obtained by perturbing the isolated eigenmodes through the interaction blocks.

\subsection{The particular case of identical resonators} \label{sec:identical_resonators}

Let us here note that in the case where $D_1$ and $D_2$ are identical particles, then, we clearly see that $\mathcal{K}^{(0)}_{11} \equiv \mathcal{K}^{(0)}_{22}$ and $\mathcal{K}^{(0)}_{12} \equiv \mathcal{K}^{(0)}_{21}$. This gives
\begin{align*}
    \mathcal K_{\mathrm{diag}}^{(0)} :=
    \begin{pmatrix}
        \mathcal K_{11}^{(0)} & 0\\
        0 & \mathcal K_{11}^{(0)}
    \end{pmatrix}
    \qquad \text{ and } \qquad
    \mathcal K_{\mathrm{off}}^{(0)} :=
    \begin{pmatrix}
        0 & \mathcal K_{12}^{(0)}\\
        \mathcal K_{12}^{(0)} & 0
    \end{pmatrix}.
\end{align*}
Then, the spectrum of $\mathcal K_{\mathrm{diag}}^{(0)}$ is obtained from the spectrum of $\mathcal K_{D_1}^{(0)}$.
If we denote by $\{\rho_n^{(1)},f_n^{(1)}\}_{n\geq1}$ the eigenpairs of the operator $\mathcal{K}^{(0)}_{11}$, the eigenvectors of $\mathcal{K}^{(0)}_{\mathrm{diag}}$ are obtained by lifting the isolated eigenvectors to $L^2(D_1)$. Namely, we define
\begin{align}
    \Phi_j^{(1)} :=
    \begin{pmatrix}
        f_j^{(1)}\\
        f_j^{(1)}
    \end{pmatrix},
    \qquad \text{ and } \qquad
    \Phi_j^{(2)} :=
    \begin{pmatrix}
        -f_j^{(1)}\\
        f_j^{(1)}
    \end{pmatrix}.
\end{align}
Then,
\begin{align}
    \mathcal K_{\mathrm{diag}}^{(0)}\Phi_j^{(1)} = \rho_j^{(1)}\Phi_j^{(1)}
    \qquad \text{ and } \qquad
    \mathcal K_{\mathrm{diag}}^{(0)}\Phi_j^{(2)} = \rho_j^{(1)}\Phi_j^{(2)}.
\end{align}

\subsection{Weak-coupling regime}

We now define the weak-coupling regime. Physically, this is the regime in which each resonator retains its isolated modal structure, while the interaction between the two particles produces only a perturbative correction. Mathematically, this means that the off-diagonal interaction is small compared with the diagonal self-interaction.

\begin{definition}[Weak-coupling regime]
We say that the two-resonator system is in the weak-coupling regime if
\begin{align}\label{eq:weakcoupling_3D}
    \| \mathcal K_{\mathrm{off}}^{(0)} \|
    \ll
    \| \mathcal K_{\mathrm{diag}}^{(0)} \|.
\end{align}
\end{definition}

\begin{remark}
    The condition \eqref{eq:weakcoupling_3D} ensures that the coupled spectral structure remains organized by the isolated resonator modes. In particular, the coupled eigenmodes remain perturbative deformations of the lifted isolated eigenfunctions.
\end{remark}

A sharper spectral version of this condition compares the size of the off-diagonal interaction with the spectral separation of the isolated resonators. In that sense, we are able to define the weak coupling regime as the perturbation around a specific eigenpair of either $\mathcal{K}^{(0)}_{ii}$, for $i=1,2$.

\begin{definition}[Spectral weak-coupling regime]
    Let $j=1,2$ and $n\geq1$. Then, 
    for the eigenpair $\{ \r_n^{(j)}, f_n^{(j)} \}$, we say that the two-resonator system is in the weak-coupling regime if
    \begin{align}\label{eq:spectralgap_3D}
        \| \mathcal K_{\mathrm{off}}^{(0)} \| < \operatorname{dist} \left( \rho_n^{(j)}, \operatorname{sp}(\mathcal K_{\mathrm{diag}}^{(0)}) \setminus \{\rho_n^{(j)}\} \right).
    \end{align}
\end{definition}

Under this condition, the interaction between the particles is smaller than the spectral gap of the uncoupled problem. Hence, the corresponding coupled eigenmodes remain perturbative deformations of the isolated modes, in accordance with standard spectral perturbation theory \cite{Kato}.

This perturbative structure is fundamental for the hotspot mechanism. In the weak-coupling regime, the resonant field is not arbitrary: it is organized by the isolated eigenmodes of the two particles. If the coupled mode induces distinct effective values on the two facing sides of the resonators, then the field must transition across the gap. Since the gap has width $\kappa$, a mean value argument implies that the gradient in the gap can become large, with an intermediate scaling of order $\kappa^{-1}$.

The weak-coupling regime is therefore the regime in which the isolated modal structure remains stable enough to create a strong field contrast, while the geometry of the narrow gap converts this contrast into magnetic localization.

\subsection{Modal structure in the weak-coupling regime}
\label{subsec:modal_structure_weak_coupling}

We now describe the modal structure of the coupled operator in the weak-coupling regime. Instead of writing the perturbative expansion separately for each lifted isolated mode, we use a notation adapted to a single isolated eigenvalue.\\
Let $\Phi\in L^2(D)$ be an $L^2$-normalized eigenfunction of the diagonal operator $\mathcal K_{\mathrm{diag}}^{(0)}$, associated with a simple isolated eigenvalue $\rho$, namely
\begin{align}
    \mathcal K_{\mathrm{diag}}^{(0)}\Phi=\rho\Phi,
    \qquad
    \|\Phi\|_{L^2(D)}=1.
\end{align}
We denote by
\begin{align}
    d_\rho
    :=
    \operatorname{dist}
    \left(
    \rho,
    \operatorname{sp}(\mathcal K_{\mathrm{diag}}^{(0)})\setminus\{\rho\}
    \right)
    >0
    \label{eq:spectral_gap_drho}
\end{align}
the spectral gap associated with $\rho$. Let $P_\rho$ be the orthogonal projection onto $\operatorname{span}\{\Phi\}$ and set $Q_\rho=I-P_\rho$. The reduced resolvent is defined by
\begin{align}
    S_\rho
    :=
    Q_\rho(\mathcal K_{\mathrm{diag}}^{(0)}-\rho I)^{-1}Q_\rho.
\end{align}
It satisfies
\begin{align}
    \|S_\rho\|_{L^2\to L^2}\leq \frac{1}{d_\rho}.
\end{align}
We define
\begin{align}
    \varepsilon_2
    :=
    \|\mathcal K_{\mathrm{off}}^{(0)}\|_{L^2\to L^2}.
\end{align}
The spectral weak-coupling regime near the eigenvalue $\rho$ is the regime in which
\begin{align}
    \frac{\varepsilon_2}{d_\rho}\ll1.
    \label{eq:spectral_wc_eps}
\end{align}
This condition means that the off-diagonal interaction is small compared with the spectral separation of the isolated mode under consideration.

Let $\Phi_\kappa=\Phi+\delta\phi$ be the perturbed eigenfunction associated with $\rho_\kappa=\rho+\delta\rho$, so that
\begin{align*}
    (\mathcal{K}_{\mathrm{diag}}^{(0)}+\mathcal{K}_{\mathrm{off}}^{(0)})\Phi_\kappa
    =
    \rho_\kappa\Phi_\kappa.
\end{align*}
We choose the phase and normalization by imposing
\begin{align*}
    \|\Phi_\kappa\|_{L^2}=1,
    \qquad
    \langle \Phi_\kappa,\Phi\rangle_{L^2}\in\mathbb R_{>0},
\end{align*}
and define
\begin{align*}
    \delta\phi:=\Phi_\kappa-\Phi.
\end{align*}

\begin{theorem}[Weak-coupling modal perturbation]
\label{thm:3D_modal_expansion}
Assume that $\rho$ is a simple isolated eigenvalue of $\mathcal K_{\mathrm{diag}}^{(0)}$ and that \eqref{eq:spectral_wc_eps} holds. Then there exists an eigenpair $(\rho_\kappa,\Phi_\kappa)$ of $\mathcal K_D^{(0)}$
such that
\begin{align}
    \|\delta\phi\|_{L^2(D)}
    \leq
    C\frac{\varepsilon_2}{d_\rho},
    \label{eq:L2_modal_perturbation}
\end{align}
and
\begin{align}
    \delta\phi
    =
    -S_\rho\mathcal K_{\mathrm{off}}^{(0)}\Phi
    +
    O_{L^2}
    \left(
        \frac{\varepsilon_2^2}{d_\rho^2}
    \right),
    \label{eq:first_order_modal_perturbation}
\end{align}
while
\begin{align}
    \delta\rho
    =
    \left\langle
    \mathcal K_{\mathrm{off}}^{(0)}\Phi,\Phi
    \right\rangle_{L^2(D)}
    +
    O\left(\frac{\varepsilon_2^2}{d_\rho}\right).
    \label{eq:eigenvalue_shift}
\end{align}
\end{theorem}

\begin{proof}
This is the standard non-degenerate perturbation expansion for a compact self-adjoint operator, applied to the decomposition
\begin{align*}
    \mathcal K_D^{(0)}
    =
    \mathcal K_{\mathrm{diag}}^{(0)}
    +
    \mathcal K_{\mathrm{off}}^{(0)}.
\end{align*}
The estimate follows from the reduced resolvent bound
\(\|S_\rho\|_{L^2\to L^2}\leq d_\rho^{-1}\).
\end{proof}

\begin{remark}
If the isolated eigenmode is supported on one resonator, for instance
\begin{align*}
    \Phi=(f_j^{(1)},0)^T,
\end{align*}
then
\begin{align*}
    \mathcal K_{\mathrm{off}}^{(0)}\Phi
    =
    (0,\mathcal K_{21}^{(0)}f_j^{(1)})^T.
\end{align*}
Hence
\begin{align*}
    \left\langle
    \mathcal K_{\mathrm{off}}^{(0)}\Phi,\Phi
    \right\rangle_{L^2(D)}
    =
    0,
\end{align*}
and the first-order eigenvalue shift vanishes. In this case,
\begin{align}
    |\delta\rho|
    \leq
    C\frac{\varepsilon_2^2}{d_\rho}.
    \label{eq:quadratic_eigenvalue_shift}
\end{align}
This is the non-degenerate two-resonator situation used below.
\end{remark}

\begin{remark}
If the isolated eigenvalue has multiplicity larger than one, the same perturbative statement must be written in terms of spectral projections onto the corresponding eigenspace, rather than individual eigenvectors. This is the natural framework for non-radial spherical modes and for identical resonators.
\end{remark}

\section{\texorpdfstring{$C^0$}{} control of the perturbation}
\label{sec:C0_control}

The mean value mechanism requires pointwise information on the coupled eigenmode. The $L^2$ perturbation estimate \eqref{eq:L2_modal_perturbation} is therefore not sufficient on its own: we must control the perturbation in $C^0(\overline D)$.

We introduce the operator norms
\begin{align}
    \varepsilon_\infty
    :=
    \|\mathcal K_{\mathrm{off}}^{(0)}\|_{L^2\to C^0},
    \qquad
    C_0
    :=
    \|\mathcal K_{\mathrm{diag}}^{(0)}\|_{L^2\to C^0}.
    \label{eq:epsinf_C0_def}
\end{align}
The Newtonian potential maps $L^2(D)$ continuously into $C^0(\overline D)$ for bounded three-dimensional domains. %In particular, the quantities in \eqref{eq:epsinf_C0_def} are finite for every fixed configuration. 
In particular, it is continuous in \(x\), and therefore the preceding
\(L^2\to L^\infty\) bounds imply corresponding \(L^2\to C^0\) bounds on compact
subsets. For the off-diagonal part, the estimates of Appendix~\ref{app:Linfty-control}
give
\begin{align}
    \varepsilon_\infty
    \leq
    \min\left\{
    \frac12
    \left(\frac{\max(R_1,R_2)}{\pi}\right)^{1/2},
    \frac{\max(|D_1|,|D_2|)^{1/2}}{4\pi\kappa}
    \right\}.
    \label{eq:epsilon_infty_bound}
\end{align}
The second bound follows directly from the pointwise estimate
\begin{align*}
    |x-y|^{-1}\leq \kappa^{-1},
    \qquad x\in D_i,\ y\in D_j,\ i\neq j.
\end{align*}
The first bound shows that $\varepsilon_\infty$ remains uniformly bounded as
$\kappa\to0$, while the second shows that the off-diagonal $L^2\to C^0$
norm becomes small as the separation grows.

\begin{proposition}[Direct $C^0$ estimate]
\label{prop:directC0}
Let $\Phi_\kappa=\Phi+\delta\phi$ and $\rho_\kappa=\rho+\delta\rho$ be as in Theorem~\ref{thm:3D_modal_expansion}. Assume that
\begin{align}
    \frac{|\delta\rho|}{|\rho|}
    \leq
    \frac12.
    \label{eq:eigenvalue_shift_small}
\end{align}
Then
\begin{align}
    \|\delta\phi\|_{C^0(\overline D)}
    \leq
    \frac{2}{|\rho|}
    \left(
    C_0\|\delta\phi\|_{L^2(D)}
    +
    \varepsilon_\infty\|\Phi\|_{L^2(D)}
    +
    \varepsilon_\infty\|\delta\phi\|_{L^2(D)}
    +
    |\delta\rho|\|\Phi\|_{C^0(\overline D)}
    \right).
    \label{eq:direct_C0_estimate}
\end{align}
\end{proposition}

\begin{proof}
The perturbed and unperturbed equations are
\begin{align*}
    (\rho+\delta\rho)(\Phi+\delta\phi)
    =
    (\mathcal{K}_{\mathrm{diag}}^{(0)} + \mathcal{K}_{\mathrm{off}}^{(0)})(\Phi+\delta\phi),
\end{align*}
and
\begin{align*}
    \rho\Phi = \mathcal{K}_{\mathrm{diag}}^{(0)}\Phi.
\end{align*}
Subtracting gives
\begin{align*}
    \rho\,\delta\phi
    =
    \mathcal{K}_{\mathrm{diag}}^{(0)}\delta\phi
    +
    \mathcal{K}_{\mathrm{off}}^{(0)}\Phi
    +
    \mathcal{K}_{\mathrm{off}}^{(0)}\delta\phi
    -
    \delta\rho\,\Phi
    -
    \delta\rho\,\delta\phi.
\end{align*}
Taking the $C^0(\overline D)$ norm and using
\begin{align*}
    \|\mathcal{K}_{\mathrm{diag}}^{(0)}u\|_{C^0}
    \leq
    C_0\|u\|_{L^2},
    \qquad
    \|\mathcal{K}_{\mathrm{off}}^{(0)}u\|_{C^0}
    \leq
    \varepsilon_\infty\|u\|_{L^2},
\end{align*}
we obtain
\begin{align*}
    \|\delta\phi\|_{C^0}
    \leq
    \frac{1}{|\rho|}
    \left(
    C_0\|\delta\phi\|_{L^2}
    +
    \varepsilon_\infty\|\Phi\|_{L^2}
    +
    \varepsilon_\infty\|\delta\phi\|_{L^2}
    +
    |\delta\rho|\|\Phi\|_{C^0}
    +
    |\delta\rho|\|\delta\phi\|_{C^0}
    \right).
\end{align*}
Since $|\delta\rho|/|\rho|\leq 1/2$, the last term can be absorbed into the left-hand side. This gives the desired estimate.
\end{proof}

\begin{corollary}[Pointwise perturbation estimate]
\label{cor:C0_control}
Assume that $\varepsilon_2\leq c\,d_\rho$, with $c>0$ sufficiently small, and that \eqref{eq:eigenvalue_shift_small} holds. Then
\begin{align}
    \|\delta\phi\|_{C^0(\overline D)}
    \leq
    C\left(
    \frac{C_0}{|\rho|}\frac{\varepsilon_2}{d_\rho}
    +
    \frac{\varepsilon_\infty}{|\rho|}
    +
    \frac{\varepsilon_\infty}{|\rho|}\frac{\varepsilon_2}{d_\rho}
    +
    \frac{C_0}{|\rho|^2}\varepsilon_2
    \right).
    \label{eq:C0_control_general}
\end{align}
\end{corollary}

\begin{proof}
We apply Proposition~\ref{prop:directC0}. Since $\|\Phi\|_{L^2}=1$ and
\begin{align*}
    \Phi=\rho^{-1}\mathcal{K}_{\mathrm{diag}}^{(0)}\Phi,
\end{align*}
we have
\begin{align*}
    \|\Phi\|_{C^0}
    \leq
    \frac{C_0}{|\rho|}.
\end{align*}
Using the $L^2$ perturbation bound
\begin{align*}
    \|\delta\phi\|_{L^2}
    \leq
    C_1\frac{\varepsilon_2}{d_j},
\end{align*}
and the eigenvalue perturbation bound
\begin{align*}
    |\delta\rho|\leq C_2\varepsilon_2,
\end{align*}
gives the result.
\end{proof}

In the non-degenerate two-resonator setting, the first-order eigenvalue correction vanishes, as observed in \eqref{eq:quadratic_eigenvalue_shift}. We therefore obtain the sharper form used in the hotspot argument.

\begin{corollary}[Sharper pointwise estimate]
\label{cor:C0_control_sharper}
Assume in addition that
\begin{align*}
    |\delta\rho|
    =
    O\left(
    \frac{\varepsilon_2^2}{d_\rho}
    \right).
\end{align*}
Then
\begin{align}
    \|\delta\phi\|_{C^0(\overline D)}
    \leq
    C\left(
    \frac{C_0}{|\rho|}\frac{\varepsilon_2}{d_\rho}
    +
    \frac{\varepsilon_\infty}{|\rho|}
    +
    \frac{\varepsilon_\infty}{|\rho|}\frac{\varepsilon_2}{d_\rho}
    +
    \frac{C_0}{|\rho|^2}\frac{\varepsilon_2^2}{d_\rho}
    \right).
    \label{eq:C0_control_sharper}
\end{align}
where $C$ is independent of $\varepsilon_2$ and $\varepsilon_\infty$.
\end{corollary}

\begin{proof}
The proof is the same direct estimate as in Corollary~\ref{cor:C0_control}, but using the sharper eigenvalue perturbation bound
\begin{align*}
    |\delta\rho|
    \leq
    C\frac{\varepsilon_2^2}{d_\rho}.
\end{align*}
Indeed, Proposition~\ref{prop:directC0} gives
\begin{align*}
    \|\delta\phi\|_{C^0(\overline D)}
    \leq
    \frac{2}{|\rho|}
    \left(
    C_0\|\delta\phi\|_{L^2(D)}
    +
    \varepsilon_\infty\|\Phi\|_{L^2(D)}
    +
    \varepsilon_\infty\|\delta\phi\|_{L^2(D)}
    +
    |\delta\rho|\|\Phi\|_{C^0(\overline D)}
    \right).
\end{align*}
Using \(\|\Phi\|_{L^2(D)}=1\),
\begin{align*}
    \|\delta\phi\|_{L^2(D)}
    \leq
    C\frac{\varepsilon_2}{d_\rho},
\end{align*}
and
\begin{align*}
    \|\Phi\|_{C^0(\overline D)}
    \leq
    \frac{C_0}{|\rho|},
\end{align*}
we obtain
\begin{align*}
    \|\delta\phi\|_{C^0(\overline D)}
    \leq
    C\left(
    \frac{C_0}{|\rho|}\frac{\varepsilon_2}{d_\rho}
    +
    \frac{\varepsilon_\infty}{|\rho|}
    +
    \frac{\varepsilon_\infty}{|\rho|}\frac{\varepsilon_2}{d_\rho}
    +
    \frac{C_0}{|\rho|^2}\frac{\varepsilon_2^2}{d_\rho}
    \right).
\end{align*}
This proves the estimate.
\end{proof}

\paragraph*{Pointwise control.}

The preceding estimate gives the pointwise control needed in the proof of Theorem~\ref{thm:hotspot}. Indeed,
\begin{align*}
    \Phi_\kappa=\Phi+\delta\phi,
\end{align*}
with
\begin{align*}
    \|\delta\phi\|_{C^0}
    \leq
    C\left(
    \frac{C_0}{|\rho|}\frac{\varepsilon_2}{d_\r}
    +
    \frac{\varepsilon_\infty}{|\rho|}
    +
    \frac{\varepsilon_\infty}{|\rho|}\frac{\varepsilon_2}{d_\r}
    +
    \frac{C_0}{|\rho|^2}\frac{\varepsilon_2^2}{d_\r}
    \right).
\end{align*}
Thus, whenever
\begin{align*}
    \frac{C_0}{|\rho|}\frac{\varepsilon_2}{d_\r}
    +
    \frac{\varepsilon_\infty}{|\rho|}
    +
    \frac{\varepsilon_\infty}{|\rho|}\frac{\varepsilon_2}{d_\r}
    +
    \frac{C_0}{|\rho|^2}\frac{\varepsilon_2^2}{d_\r}
    \ll 1,
\end{align*}
the coupled eigenmode remains uniformly close to the isolated eigenmode. In particular, if the isolated mode has a non-vanishing boundary contrast at the closest points $x_1\in\partial D_1$ and $x_2\in\partial D_2$, then this contrast persists under the perturbation. The mean value argument in the proof of Theorem~\ref{thm:hotspot} then gives the gradient field amplification estimate.

\section{Mean value mechanism and gradient field amplification}
\label{sec:MVT}

We now explain how the pointwise perturbation control obtained in Section~\ref{sec:C0_control} produces gradient hotspots. Let $x_1\in\partial D_1$ and $x_2\in\partial D_2$ be the closest boundary points, so that
\begin{align*}
    |x_2-x_1|=\kappa.
\end{align*}
The mechanism is geometric: if a weakly coupled mode keeps a non-vanishing contrast between the two facing boundary points, then the field must transition across a region of width $\kappa$.

\begin{theorem}[Gradient hotspot formation in the weak-coupling regime]
\label{thm:hotspot}
Let $\Phi$ be an $L^2$-normalized isolated eigenmode of $\mathcal K_{\mathrm{diag}}^{(0)}$, associated with a simple isolated eigenvalue $\rho$, and let $\Phi_\kappa=\Phi+\delta\phi$ be the corresponding coupled eigenmode of $\mathcal K_D^{(0)}$. Assume that the weak-coupling condition
\begin{align*}
    \frac{\varepsilon_2}{d_\rho}\ll1
\end{align*}
holds and that the pointwise perturbation satisfies the estimate \eqref{eq:C0_control_sharper}. Assume moreover that the isolated mode has a non-vanishing contrast at the closest boundary points:
\begin{align}
    |\Phi(x_2)-\Phi(x_1)|
    \geq
    c_*
    >
    0.
    \label{eq:isolated_boundary_contrast}
\end{align}
If
\begin{align}
    \|\Phi_\kappa-\Phi\|_{C^0(\overline D)}
    \leq
    \frac{c_*}{4},
    \label{eq:small_C0_for_contrast}
\end{align}
then there exists a point $y$ on the segment joining $x_1$ and $x_2$ such that
\begin{align}
    |\nabla \Phi_\kappa(y)|
    \geq
    \frac{c_*}{2\kappa}.
    \label{eq:gradient_hotspot_bound}
\end{align}
\end{theorem}

\begin{proof}
By the triangle inequality,
\begin{align*}
    |\Phi_\kappa(x_2)-\Phi_\kappa(x_1)|
    &\geq
    |\Phi(x_2)-\Phi(x_1)|
    -
    |\delta\phi(x_2)|
    -
    |\delta\phi(x_1)| \\
    &\geq
    c_*
    -
    2\|\delta\phi\|_{C^0(\overline D)}.
\end{align*}
Using \eqref{eq:small_C0_for_contrast}, we obtain
\begin{align*}
    |\Phi_\kappa(x_2)-\Phi_\kappa(x_1)|
    \geq
    \frac{c_*}{2}.
\end{align*}
Applying the mean value theorem along the segment joining $x_1$ and $x_2$, there exists a point $y$ on this segment such that
\begin{align*}
    |\nabla \Phi_\kappa(y)|
    \geq
    \frac{|\Phi_\kappa(x_2)-\Phi_\kappa(x_1)|}{|x_2-x_1|}
    \geq
    \frac{c_*}{2\kappa}.
\end{align*}
This proves the result.
\end{proof}

The theorem shows that gradient amplification follows from two ingredients. The first is modal: the isolated mode must create a nonzero contrast between the two facing sides of the resonators. The second is perturbative: the coupled mode must remain pointwise close enough to the isolated mode for this contrast to survive. In that regime, the narrow geometry converts the persistent contrast into a gradient of order $\kappa^{-1}$.

\subsection{Physical interpretation}

The overall picture is therefore the following. For moderately small gaps, the resonators remain weakly coupled. The dominant resonant mode retains the structure of an isolated resonator mode, and the field values on the two facing boundaries remain significantly different. The narrow gap then converts this modal contrast into a large gradient field through the mean value mechanism.

For sufficiently small gaps, however, the interaction between the resonators becomes too strong for the perturbative weak-coupling description to remain valid. The resonant mode reorganizes across the coupled structure, the boundary contrast decreases, and the inverse-gap amplification mechanism saturates.

Thus, gradient hotspots in the present three-dimensional dielectric setting should be understood as strong but bounded concentration phenomena generated by weakly coupled resonant interactions. The apparent blow-up observed in the intermediate regime is therefore a pre-asymptotic weak-coupling effect rather than a genuine singularity of the governing equations.

\section{Numerical experiments}
\label{sec:Numerics}

This section provides numerical evidence for the weak-coupling mechanism developed in the preceding sections. Our objectives are threefold. First, we verify the existence of a regime in which the interaction between the two resonators remains perturbative relative to their isolated modal structure. Second, we confirm that this weak-coupling regime produces the inverse-gap amplification predicted by the mean value analysis. Third, we investigate the transition to strong coupling and the associated saturation of the gradient field.

To carry out these computations, we consider two spherical dielectric resonators and exploit the explicit spectral decomposition of the Newtonian potential on a ball. This provides a natural finite-dimensional basis for constructing a projected interaction operator and tracking its spectral evolution as the gap size varies. The resulting framework allows us to compute the coupled eigenmodes, reconstruct the corresponding physical fields, and evaluate the weak-coupling, mean value, and spectral diagnostics introduced in the theoretical analysis.

\subsection{Spectral basis on a sphere}

We first recall the spectral result used in the numerical construction. The following theorem is due to Anderson and Khavinson \cite{anderson1992spectral}. We state it for the Newtonian potential on the unit ball. The corresponding result for a ball of radius $R$ follows by scaling.

\begin{theorem}[Anderson--Khavinson spectral decomposition]
\label{thm:anderson_khavinson}
Let $B\subset \mathbb R^3$ be the unit ball, and let
\begin{align*}
    \mathcal N_B[u](x) := \frac{1}{4\pi} \int_B \frac{u(y)}{|x-y|}\upd y
\end{align*}
be the Newtonian potential on $B$. Then the spectrum of $\mathcal N_B$ is given by
\begin{align*}
    \operatorname{sp}(\mathcal N_B) =
    \left\{ \rho_{s,n} = \frac{1}{j_{s-\frac12,n}^2}; \ s=0,1,\dots,\quad n=1,2,\dots \right\},
\end{align*}
where $j_{\nu,n}$ denotes the $n$-th positive zero of the Bessel function $J_\nu$.

For $s\geq1$, each eigenvalue $\rho_{s,n}$ has multiplicity $2s+1$, and the corresponding eigenspace is spanned by functions of the form
\begin{align*}
    f_{s,m,n}(r,\theta,\phi) = r^{-\frac12} J_{s+\frac12} \bigl(j_{s-\frac12,n}r\bigr) S_{s,m}(\theta,\phi), \qquad m=-s,\dots,s,
\end{align*}
where $S_{s,m}$ are spherical harmonics.

For $s=0$, the eigenvalues are simple, and the corresponding eigenfunctions are spherically symmetric. They are given by
\begin{align*}
    f_{0,n}(r) = C_n r^{-\frac12} J_{\frac12} \bigl(j_{-\frac12,n}r\bigr) = C_n \frac{\sin(j_{-\frac12,n}r)}{r}, \qquad n=1,2,\dots,
\end{align*}
where $C_n$ is a normalization constant. 
\end{theorem}

In this work, our sign convention for the static Lippmann--Schwinger operator is
\begin{align*}
    \mathcal K_B^{(0)}[u](x) = -\frac{1}{4\pi} \int_B \frac{u(y)}{|x-y|} \upd y = -\mathcal N_B[u](x).
\end{align*}
Therefore, the eigenfunctions are the same as in Theorem~\ref{thm:anderson_khavinson}, while the eigenvalues of $\mathcal K_B^{(0)}$ are the negatives of the Anderson--Khavinson eigenvalues.

For a ball of radius $R$, the eigenvalues scale by $R^2$. Thus, with our sign convention,
\begin{align*}
    \rho_{s,n}^{(\rho)} = -\frac{R^2}{j_{s-\frac12,n}^2}.
\end{align*}

\subsection{Finite-dimensional projected model}

For each isolated sphere $D_i$, $i=1,2$, we retain the first $N$ eigenmodes $\left\{f_1^{(i)},\dots,f_N^{(i)}\right\}$, with corresponding eigenvalues $\left\{\rho_1^{(i)},\dots,\rho_N^{(i)}\right\}$. The lifted basis functions are
\begin{align*}
    \Phi_n^{(1)} =
    \begin{pmatrix}
        f_n^{(1)}\\0
    \end{pmatrix},
    \qquad
    \Phi_n^{(2)} =
    \begin{pmatrix}
        0\\f_n^{(2)}
    \end{pmatrix}.
\end{align*}
We project the coupled operator $\mathcal K_D^{(0)}$ onto the finite-dimensional space
\begin{align*}
    \mathcal{V}_N = \operatorname{span} \left\{ \Phi_1^{(1)},\dots,\Phi_N^{(1)}, \Phi_1^{(2)},\dots,\Phi_N^{(2)} \right\}.
\end{align*}
In this basis, the projected matrix takes the block form
\begin{align}
    \mathbb K_N(\kappa) =
    \begin{pmatrix}
        \mathbb K_{11} & \mathbb K_{12}(\kappa)\\
        \mathbb K_{21}(\kappa) & \mathbb K_{22}
    \end{pmatrix}.
\end{align}
The diagonal blocks are
\begin{align}
    \mathbb K_{11} = \operatorname{diag} \left( \rho_1^{(1)},\dots,\rho_N^{(1)} \right)
    \qquad \text{ and } \qquad
    \mathbb K_{22} = \operatorname{diag} \left( \rho_1^{(2)},\dots,\rho_N^{(2)} \right),
\end{align}
while the off-diagonal entries are computed from
\begin{align}
    \Big(\mathbb K_{12}(\kappa)\Big)_{mn} = \left\langle \mathcal K_{12}^{(0)} f_n^{(2)}, f_m^{(1)} \right\rangle_{L^2(D_1)}
    \qquad \text{ and } \qquad
    \Big(\mathbb K_{21}(\kappa)\Big)_{mn} = \left\langle \mathcal K_{21}^{(0)} f_n^{(1)}, f_m^{(2)} \right\rangle_{L^2(D_2)}.
\end{align}
Unless otherwise stated, we use $N=15$ for each resonator.

The coupled eigenvalues and eigenvectors are obtained by diagonalizing $\mathbb K_N(\kappa)$. The physical field is then reconstructed from the projected eigenvectors using the corresponding isolated eigenmodes. The field and its gradient are evaluated using the closed-form Newtonian potential formulas inside and outside the spheres. This is the reconstruction used in the hotspot visualizations of Figure~\ref{fig:3d_radial_hotspot}.

\subsection{Representative radial mode}

As a first visualization of the reconstructed fields, we consider the radial mode $(s,m,n)=(0,0,0)$. This mode provides a representative example of the gradient hotspot mechanism: the field remains continuous across the two resonators, while the normalized gradient magnitude becomes localized in the gap region. Figure~\ref{fig:3d_radial_hotspot} shows both a near-gap view and a full-sphere view, together with three-dimensional slice visualizations of the same mode.

\begin{figure}
    \centering

    \begin{subfigure}[t]{0.8\textwidth}
        \centering
        \includegraphics[width=\textwidth]{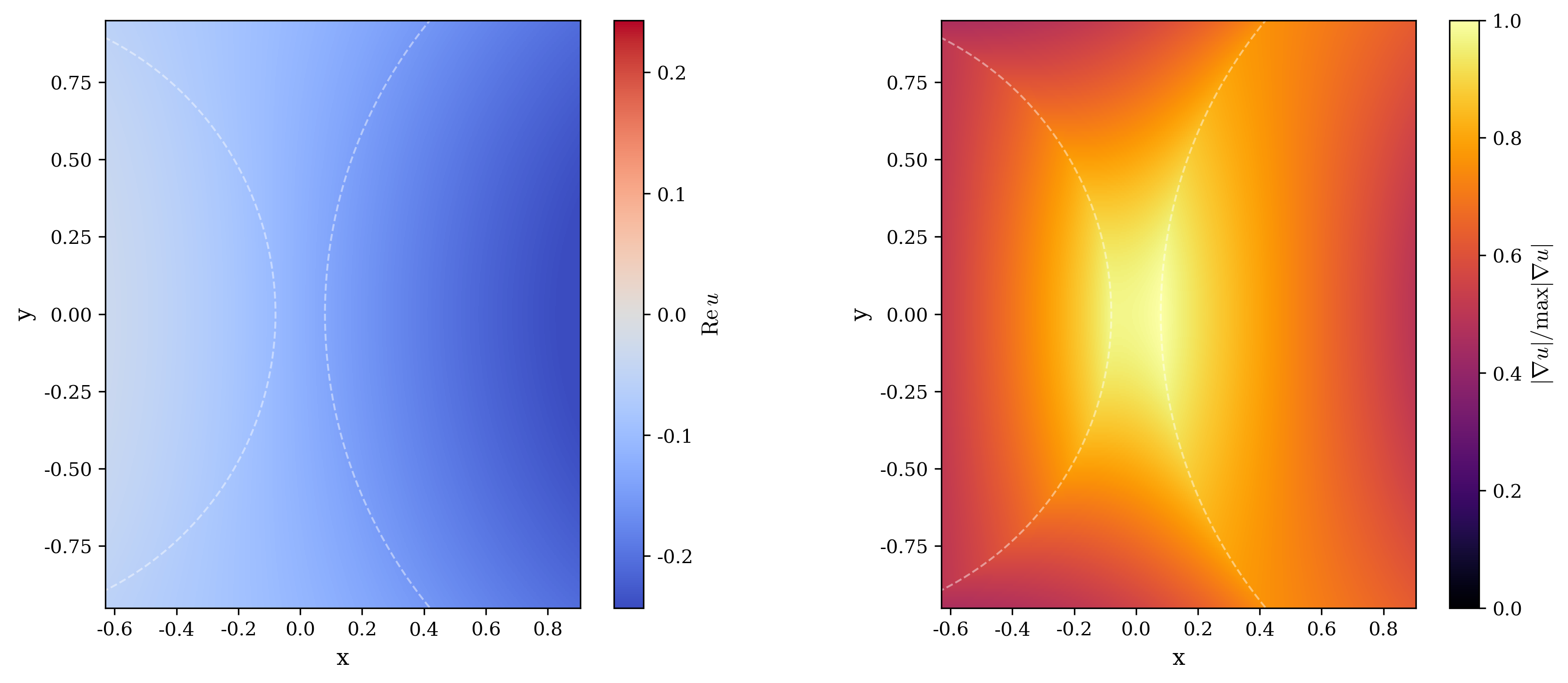}
        \subcaption{Near-gap view.}
    \end{subfigure}

    \vspace{1em}

    \begin{subfigure}[t]{0.92\textwidth}
        \centering
        \includegraphics[width=\textwidth]{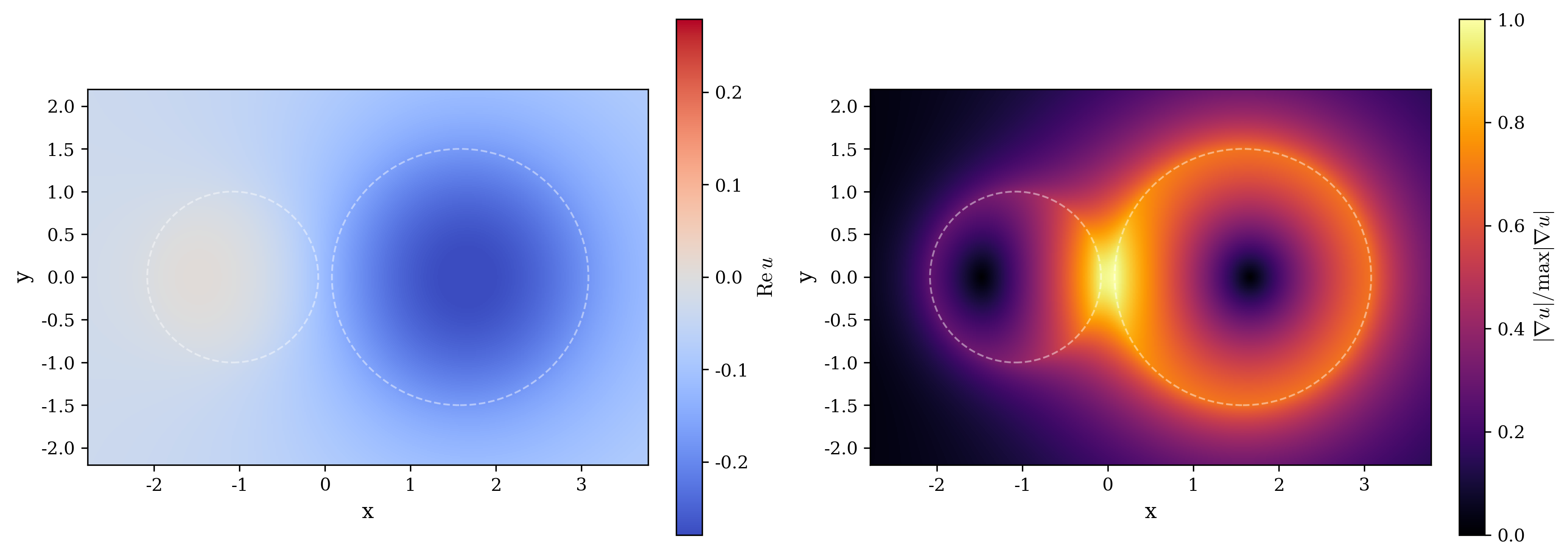}
        \subcaption{Full-sphere view.}
    \end{subfigure}
    
    \vspace{1em}

    \begin{subfigure}[t]{0.45\textwidth}
        \centering
        \includegraphics[width=\linewidth]{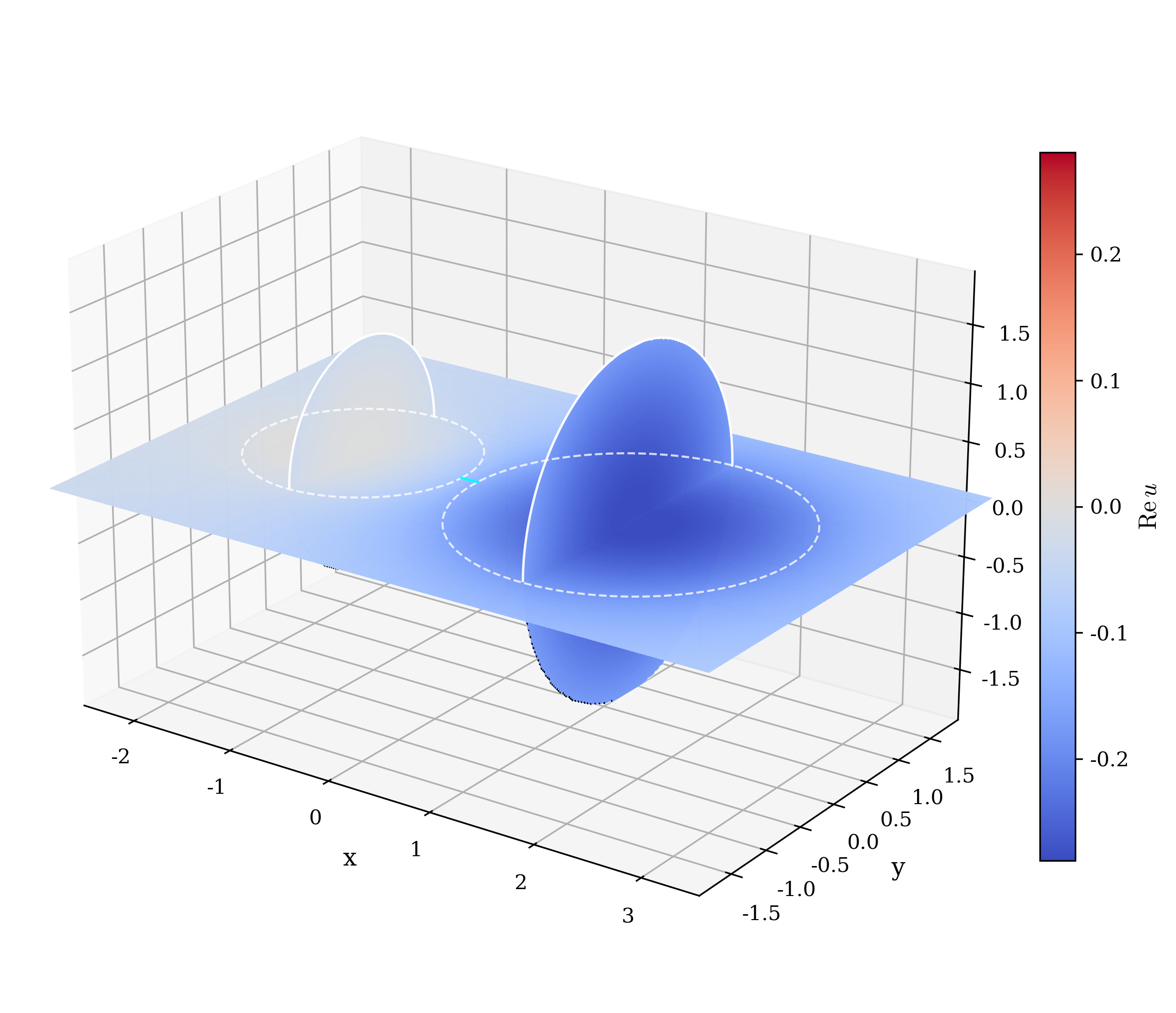}
        \subcaption{Three-dimensional slice visualization of $\Re u$.}
    \end{subfigure}
    \hfill
    \begin{subfigure}[t]{0.45\textwidth}
        \centering
        \includegraphics[width=\linewidth]{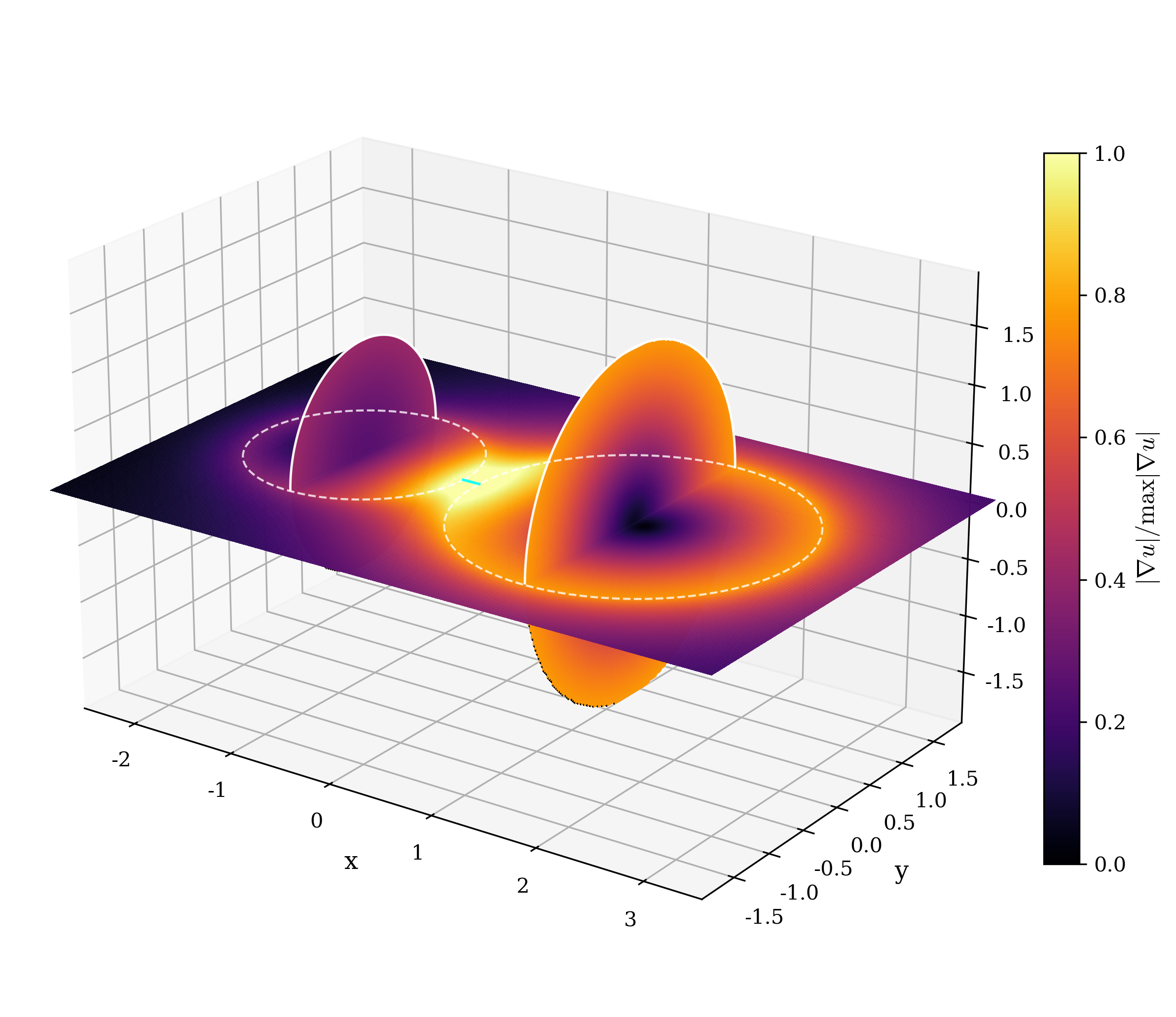}
        \subcaption{Three-dimensional slice visualization of the normalized gradient magnitude.}
    \end{subfigure}
    
    \caption{Gradient hotspot formation for the radial mode $(s,m,n)=(0,0,0)$. The field remains continuous across the resonator boundaries, while the normalized gradient magnitude becomes strongly localized near the gap. The near-gap and full-sphere views show the two-dimensional cross-section, while the lower panels provide three-dimensional slice visualizations of the same reconstructed mode.}
    \label{fig:3d_radial_hotspot}
\end{figure}

\subsection{Numerical identification of the weak-coupling regime}
\label{subsec:numerical identification 3D}

We now explain how the weak-coupling regime is identified numerically. We define the block-norm ratio
\begin{align}\label{eq:ratio_blocks_3D}
    \mathcal R(\kappa) :=
    \frac{ \|\mathbb K_{\mathrm{off}}(\kappa)\|_F }{ \|\mathbb K_{\mathrm{diag}}\|_F },
\end{align}
where
\begin{align*}
    \mathbb K_{\mathrm{diag}} =
    \begin{pmatrix}
        \mathbb K_{11} & 0\\
        0 & \mathbb K_{22}
    \end{pmatrix},
    \qquad
    \mathbb K_{\mathrm{off}}(\kappa) =
    \begin{pmatrix}
        0 & \mathbb K_{12}(\kappa)\\
        \mathbb K_{21}(\kappa) & 0
    \end{pmatrix}.
\end{align*}
The weak-coupling regime corresponds to
\begin{align}
    \mathcal R(\kappa)\ll1.
\end{align}
We also use the spectral-gap diagnostic around a fixed $n\geq1$ eigenpair $\{ \r_n^{(1)}, f_n^{(1)} \}$, as stated in \eqref{eq:spectralgap_3D}, given by
\begin{align}\label{eq:spectral_gap_ratio_3D}
    \mathcal S(\kappa) :=
    \frac{ \|\mathbb K_{\mathrm{off}}(\kappa)\|_F }{ \operatorname{dist} \left( \rho_n^{(1)}, \operatorname{sp}(\mathcal K_{\mathrm{diag}}^{(0)}) \setminus \{\rho_n^{(1)}\} \right) },
\end{align}
%where $\rho_n^{(1)}$ and $\rho_m^{(2)}$ are the isolated eigenvalues associated with the mode family under consideration. 
The perturbative spectral regime corresponds to
\begin{align}
    \mathcal S(\kappa)<1.
\end{align}

Figure~\ref{fig:weak_coupling_diagnostics_3D} displays the weak-coupling diagnostics in the scaled gap variable $\kappa/\delta$, for several values of the subwavelength parameter $\delta k_0$. The left panel shows the block-norm ratio $\mathcal R(\kappa)$ which measures the strength of the interaction between the two resonators relative to their isolated self-interaction. The right panel shows the spectral diagnostic $\mathcal S(\kappa)$ which compares the off-diagonal interaction with the spectral separation of the uncoupled resonators. In both panels, the curves for different values of $\delta k_0$ collapse when plotted against $\kappa/\delta$. This indicates that the relevant geometric parameter controlling the breakdown of weak coupling is not the absolute gap size $\kappa$, but rather the gap measured relative to the particle scale. For large values of $\kappa/\delta$, both diagnostics remain small, showing that the off-diagonal interaction is perturbative and that the coupled modes are well approximated by perturbations of the isolated resonator modes. As $\kappa/\delta$ decreases, the diagnostics increase and eventually become comparable to one, marking the onset of strong coupling. The dashed vertical lines indicate the corresponding transition scales, where the interaction ceases to be small relative either to the diagonal self-interaction or to the relevant spectral separation.

\begin{figure}
    \centering
    \includegraphics[width=0.87\linewidth]{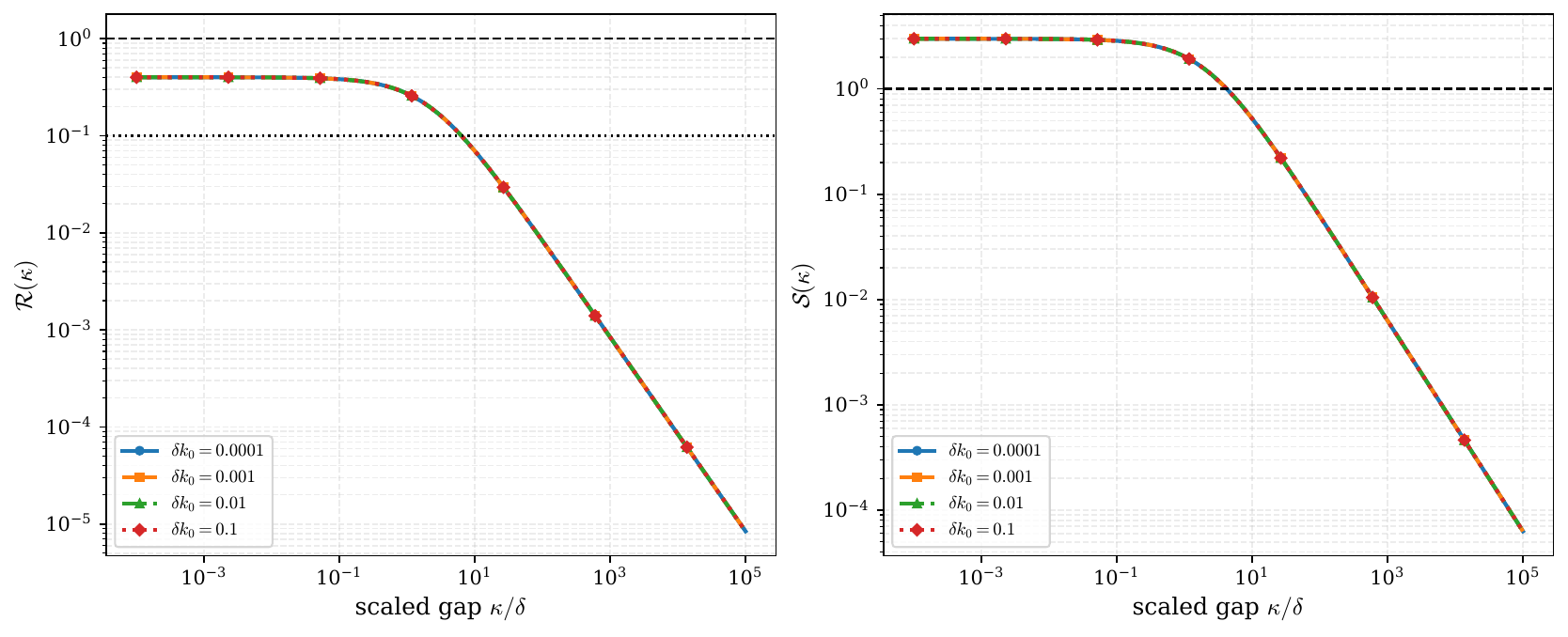}
    \caption{Weak-coupling diagnostics in three dimensions. Left: the ratio $\mathcal R(\kappa)$ measures the relative strength of the inter-particle coupling compared with the isolated self-interaction. 
    Right: the ratio $\mathcal S(\kappa)$ compares the interaction strength with the spectral separation of the isolated resonators. The dashed vertical lines mark the transition values at which the off-diagonal interaction becomes significant, and the threshold $\mathcal S(\kappa)=1$ indicates when the interaction is no longer perturbative relative to the selected spectral gap.}
    \label{fig:weak_coupling_diagnostics_3D}
\end{figure}

\begin{remark}
The collapse of the curves in Figure~\ref{fig:weak_coupling_diagnostics_3D} is a consequence of the scaling used in the numerical construction. The physical distance between the resonators is controlled by the dimensionless ratio $\kappa/\delta$, where $\delta$ is the characteristic particle size. Hence, after rescaling the geometry by $\delta$, configurations with the same value of $\kappa/\delta$ are geometrically similar. Since the static interaction operator $\mathcal K_D^{(0)}$ is homogeneous under spatial dilations, both the diagonal self-interaction and the off-diagonal interaction scale in the same way with $\delta$. Their ratio $\mathcal R(\kappa)$ is therefore primarily a function of the scaled gap $\kappa/\delta$. The same reasoning applies to the spectral diagnostic $\mathcal S(\kappa)$, since the isolated eigenvalues and their spectral separations scale with the same power of $\delta$ as the interaction terms. This explains the collapse of the curves when plotted against the scaled gap $\kappa/\delta$.
\end{remark}

\subsection{Numerical verification of the pointwise control}
\label{subsec:numerical_C0_control}

We now verify numerically the pointwise perturbation estimate proved in Section~\ref{sec:C0_control} and used in Theorem~\ref{thm:hotspot}. For the weakly coupled mode considered in the main numerical experiments, we compute the four contributions
\begin{align*}
    T_1= \frac{C_0}{|\rho|}\frac{\varepsilon_2}{d_\rho},
    \qquad
    T_2= \frac{\varepsilon_\infty}{|\rho|},
    \qquad
    T_3= \frac{\varepsilon_\infty}{|\rho|}\frac{\varepsilon_2}{d_\rho},
    \qquad
    T_4= \frac{C_0}{|\rho|^2}\frac{\varepsilon_2^2}{d_\rho},
\end{align*}
which appear in the estimate
\begin{align*}
    \|\Phi_\kappa-\Phi\|_{C^0}
    \lesssim
    T_1+T_2+T_3+T_4.
\end{align*}
The computation is performed in the same projected basis used throughout Section~\ref{sec:Numerics}, retaining the first fifteen modes on each resonator.

\begin{figure}[ht]
    \centering
    \begin{subfigure}[t]{0.48\linewidth}
        \centering
        \includegraphics[width=\linewidth]{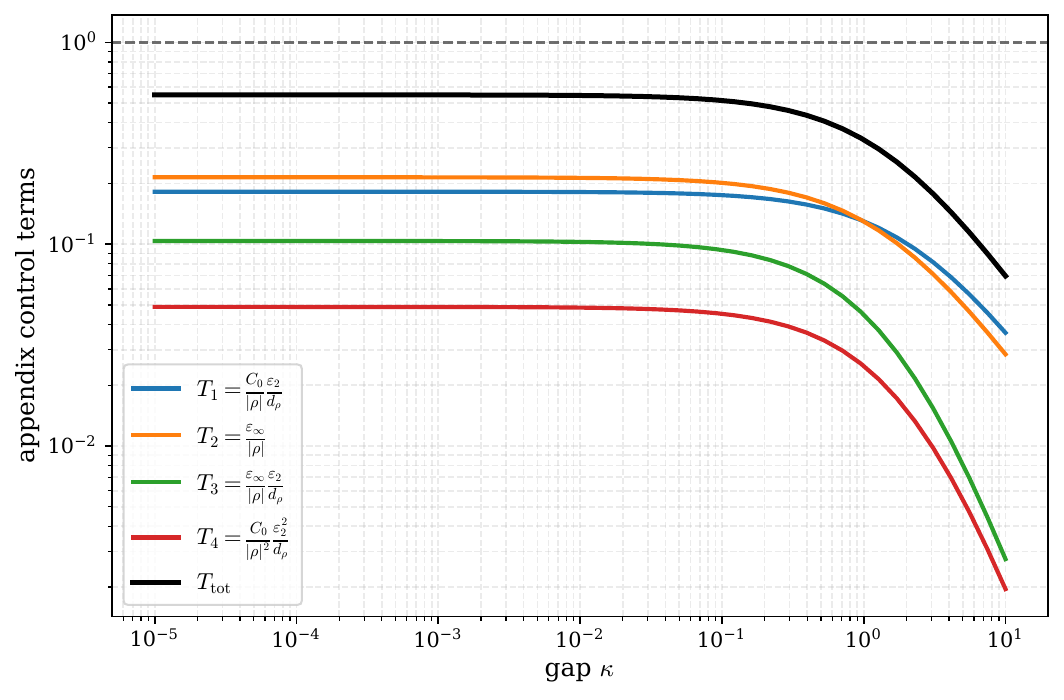}
        \subcaption{Individual contributions to the pointwise control estimate.}
        \label{fig:Linfty_control_terms}
    \end{subfigure}
    \hfill
    \begin{subfigure}[t]{0.48\linewidth}
        \centering
        \includegraphics[width=\linewidth]{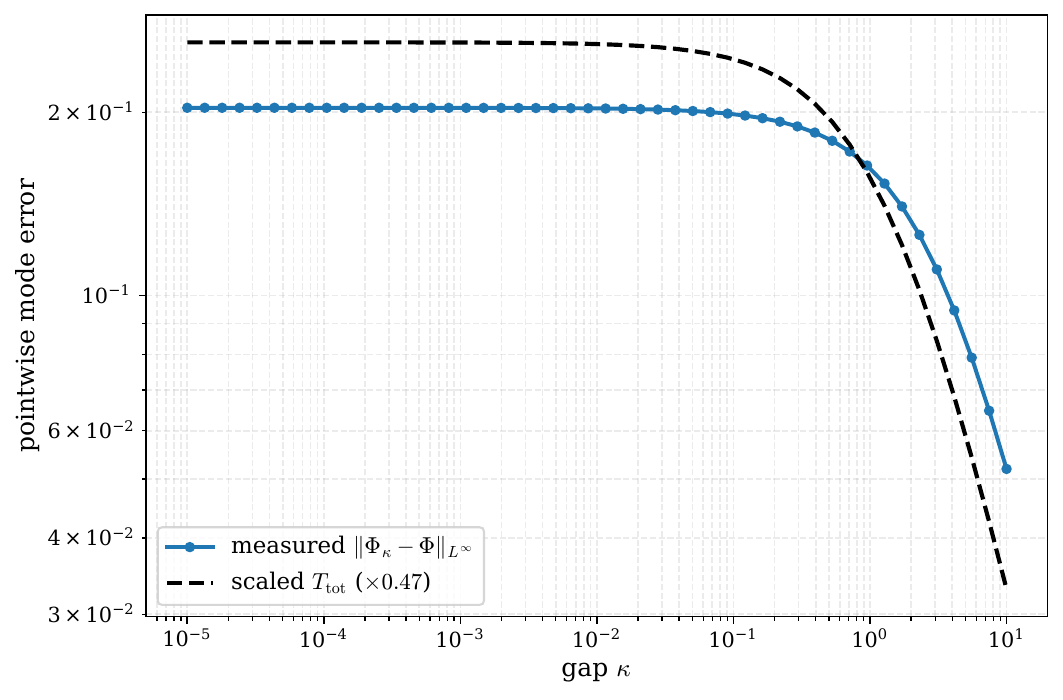}
        \subcaption{Measured pointwise perturbation compared with the predicted control quantity.}
        \label{fig:Linfty_error_vs_control}
    \end{subfigure}

    \caption{Numerical verification of the pointwise control estimate. Panel~\subref{fig:Linfty_control_terms} shows the four terms entering the bound for $\|\Phi_\kappa-\Phi\|_{C^0}$, together with their sum. Panel~\subref{fig:Linfty_error_vs_control} compares the measured pointwise perturbation with the predicted control quantity. The diagnostics show that the coupled mode remains pointwise close to the isolated mode in the weak-coupling regime, while the perturbation increases as the gap closes.}
    \label{fig:Linfty_control_diagnostics}
\end{figure}

Figure~\ref{fig:Linfty_control_diagnostics} confirms the behavior predicted by the pointwise estimate. In Figure~\ref{fig:Linfty_control_terms}, the total control quantity remains bounded throughout the weak-coupling range and increases only gradually as the particles approach each other. This reflects the fact that the off-diagonal interaction becomes stronger as the gap closes, but remains perturbative over the range where the weak-coupling description is valid.

Figure~\ref{fig:Linfty_error_vs_control} compares the measured pointwise perturbation of the coupled eigenmode with the control quantity predicted by the estimate. The two curves follow the same trend: the perturbation remains small for larger gaps, where the coupled mode is close to the isolated resonator mode, and increases as $\kappa$ decreases. This numerical check supports the use of the $C^0$ estimate in Theorem~\ref{thm:hotspot}: the boundary contrast used in the mean value argument persists precisely in the regime where the pointwise perturbation remains controlled.

\subsection{Modal dependence of the gradient hotspot formation}\label{sec:modal_dependence}

The hotspot mechanism described in Section~\ref{sec:MVT} depends strongly on the modal structure of the resonant field. While the narrow gap provides the geometric setting in which concentration may occur, the existence and intensity of a gradient hotspot are determined by the particular weakly coupled mode under consideration.

Figure~\ref{fig:modal_dependence_hotspots} illustrates this dependence for three representative modes. In each panel, we show a horizontal cross-section through the gap region, displaying the real part of the field together with the normalized gradient magnitude. This allows us to compare directly whether the modal structure creates a significant field contrast across the gap.

For the mode $(s,m,n)=(2,0,0)$, shown in Figure~\ref{fig:mode200_hotspot}, the field develops a clear contrast between the two resonators near the points of closest approach. As predicted by the mean value mechanism, this contrast forces a rapid transition across the narrow gap and produces a localized enhancement of the gradient.

A similar behavior is observed for the higher-order mode $(s,m,n)=(4,0,1)$ in Figure~\ref{fig:mode401_hotspot}. Although the internal modal structure is more oscillatory, a significant contrast persists across the gap, leading again to strong localization of the normalized gradient magnitude.

In contrast, Figure~\ref{fig:mode311_hotspot} shows the mode $(s,m,n)=(3,1,1)$. In this case, the field values on the two sides of the gap remain comparatively close, and no significant transition is imposed across the narrow region. As a result, the gradient remains much less concentrated near the gap, and no pronounced hotspot is observed.

These examples demonstrate that gradient hotspot formation is intrinsically modal. The narrow geometry alone is not sufficient to generate strong concentration. Rather, hotspots emerge for those weakly coupled modes whose structure induces a persistent field contrast between the opposing sides of the resonators.

\begin{figure}
    \centering

    \begin{subfigure}[t]{0.32\linewidth}
        \centering
        \includegraphics[width=\linewidth]{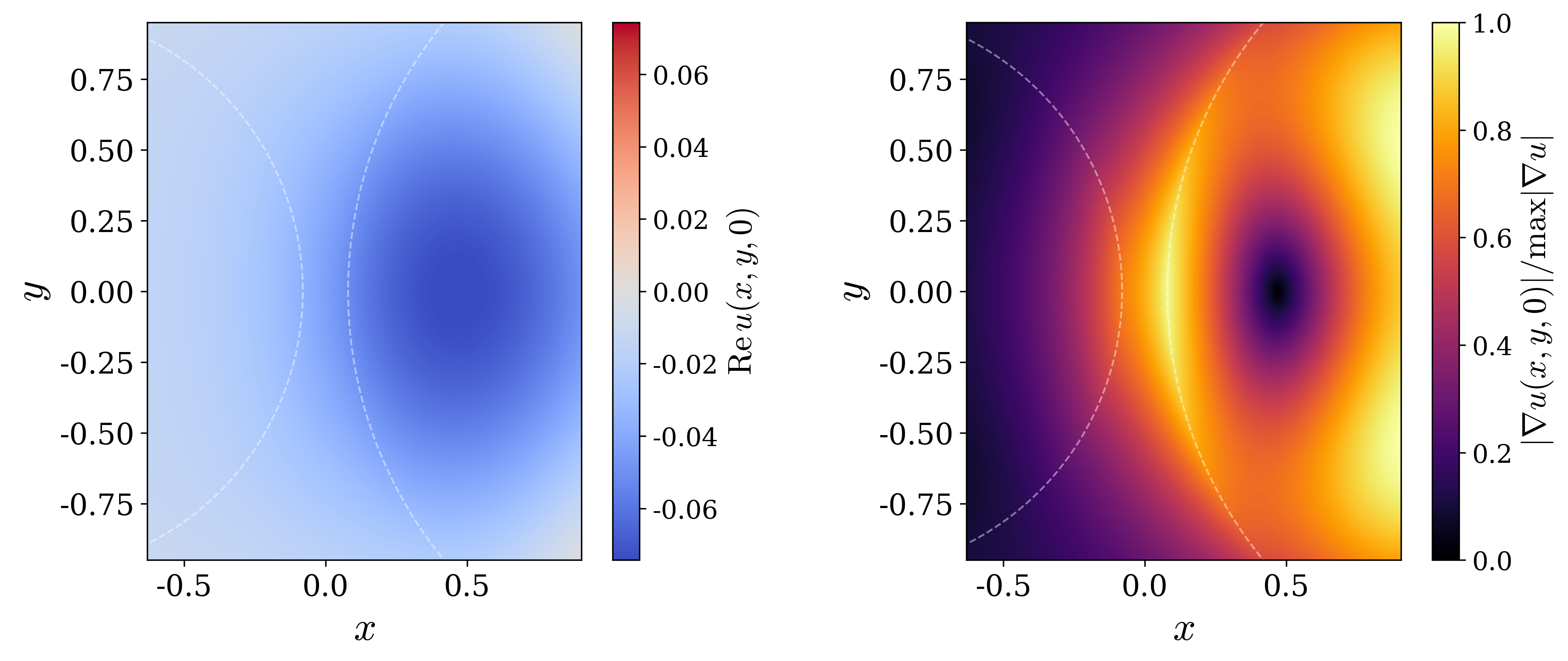}
        \subcaption{Mode $(s,m,n)=(2,0,0)$.}
        \label{fig:mode200_hotspot}
    \end{subfigure}
    \hfill
    \begin{subfigure}[t]{0.32\linewidth}
        \centering
        \includegraphics[width=\linewidth]{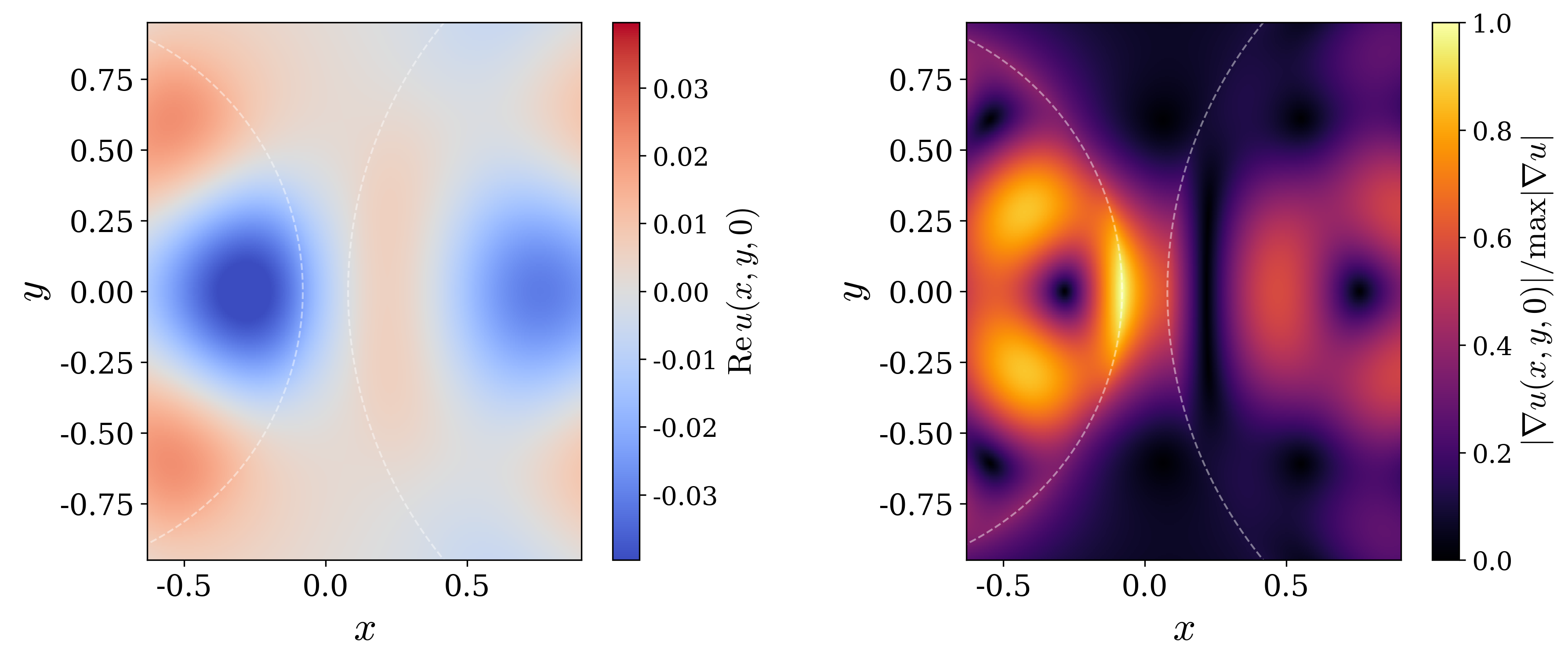}
        \subcaption{Mode $(s,m,n)=(4,0,1)$.}
        \label{fig:mode401_hotspot}
    \end{subfigure}
    \hfill
    \begin{subfigure}[t]{0.32\linewidth}
        \centering
        \includegraphics[width=\linewidth]{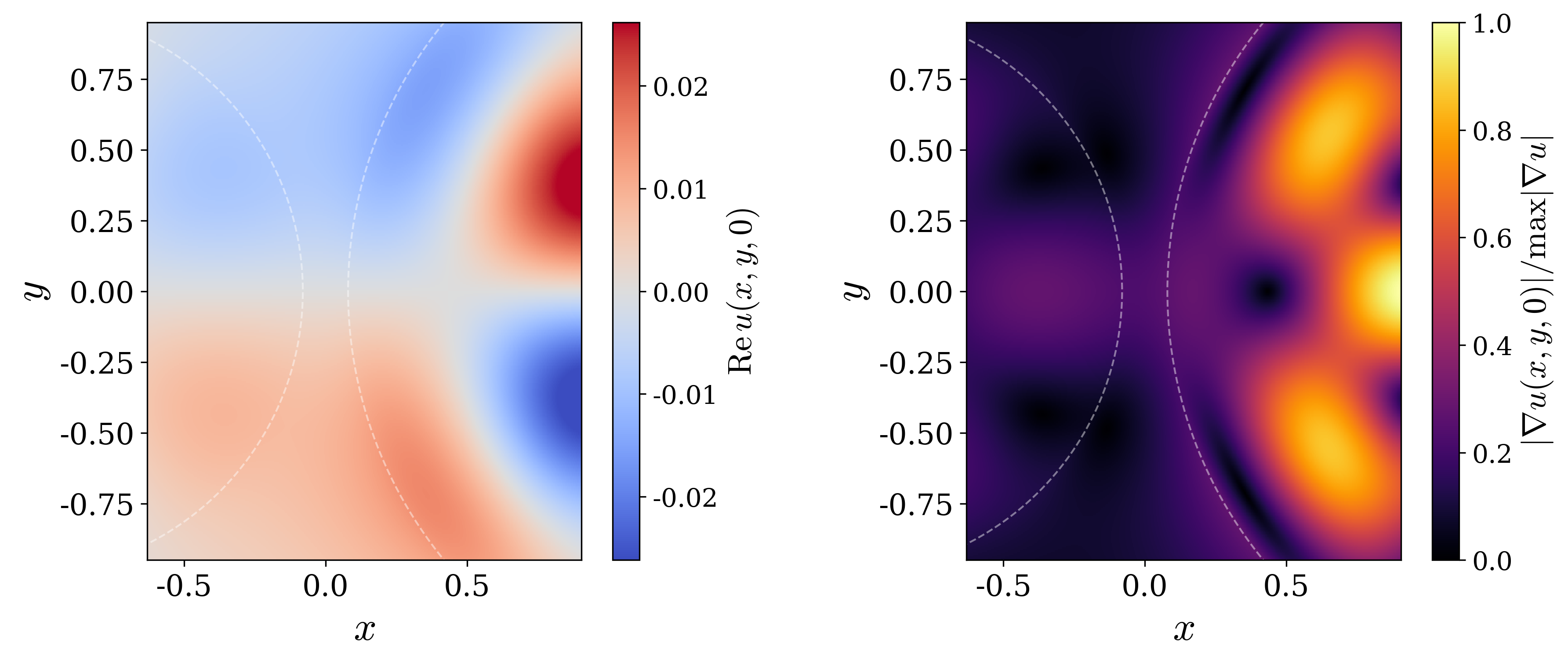}
        \subcaption{Mode $(s,m,n)=(3,1,1)$.}
        \label{fig:mode311_hotspot}
    \end{subfigure}

    \caption{Modal dependence of gradient hotspot formation. Each panel shows a horizontal cross-section through the gap region, comparing the field and the normalized gradient magnitude for a different mode. The modes $(2,0,0)$ and $(4,0,1)$ generate a significant contrast across the gap and produce localized gradient hotspots, while the mode $(3,1,1)$ does not create a comparable contrast and therefore does not produce a pronounced hotspot.}
    \label{fig:modal_dependence_hotspots}
\end{figure}

\subsection{The case of identical resonators}

We conclude the numerical experiments by considering the case of two identical spheres. As discussed in Section~\ref{sec:identical_resonators}, the isolated eigenvalues are then degenerate, and the relevant weak-coupling modes are obtained by taking symmetric and antisymmetric combinations inside the corresponding two-dimensional eigenspaces. The purpose of the present subsection is to illustrate this splitting numerically and to compare the resulting gap fields.

Figure~\ref{fig:identical_main_gap} shows the first radial mode $(s,m,n)=(0,0,0)$. The symmetric combination produces a smooth transition between the two resonators and only a moderate gradient in the gap. By contrast, the antisymmetric combination creates a sign change across the gap and generates a clear magnetic hot-spot. This is the same mean value mechanism identified above, now occurring inside a degenerate eigenspace.

\begin{figure}[ht]
    \centering
    \begin{subfigure}[t]{0.48\linewidth}
        \centering
        \includegraphics[width=\linewidth]{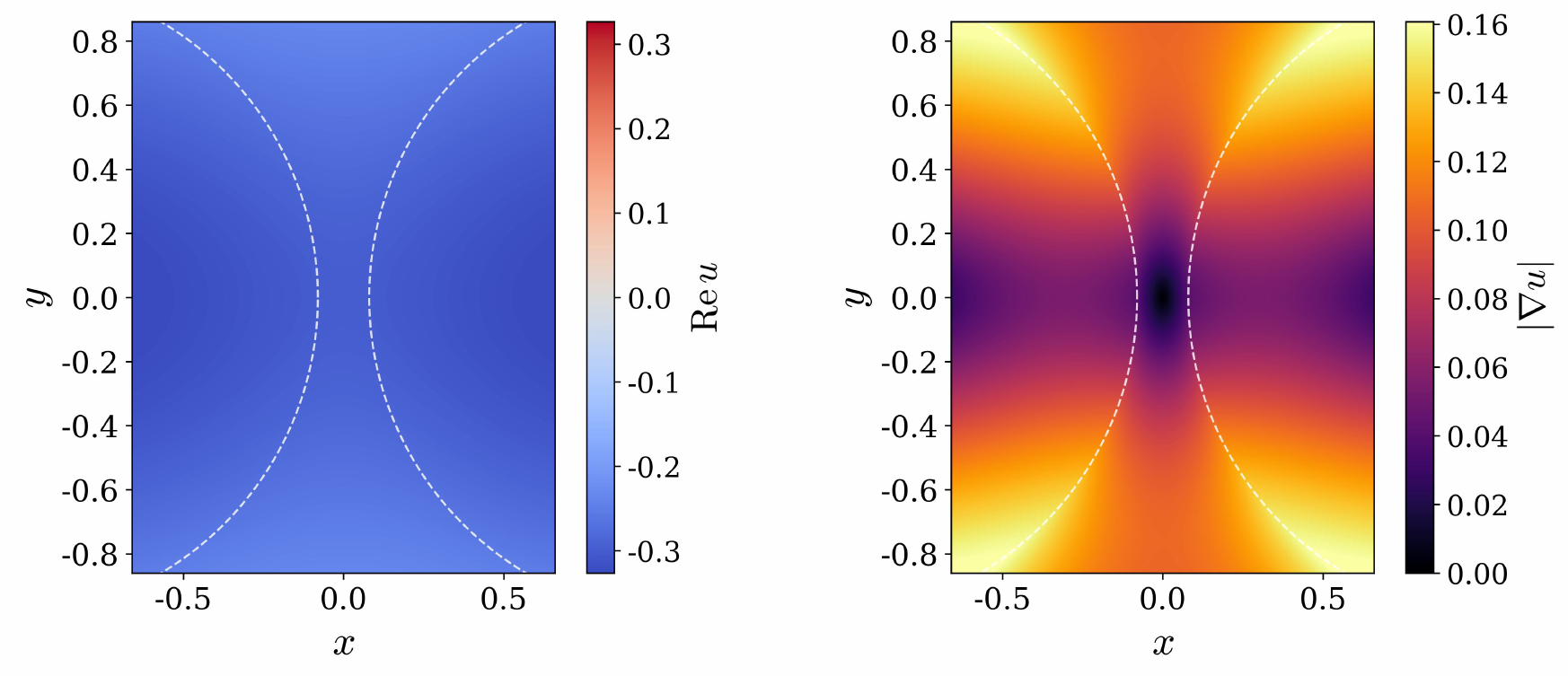}
        \subcaption{Symmetric combination.}
    \end{subfigure}
    \hfill
    \begin{subfigure}[t]{0.48\linewidth}
        \centering
        \includegraphics[width=\linewidth]{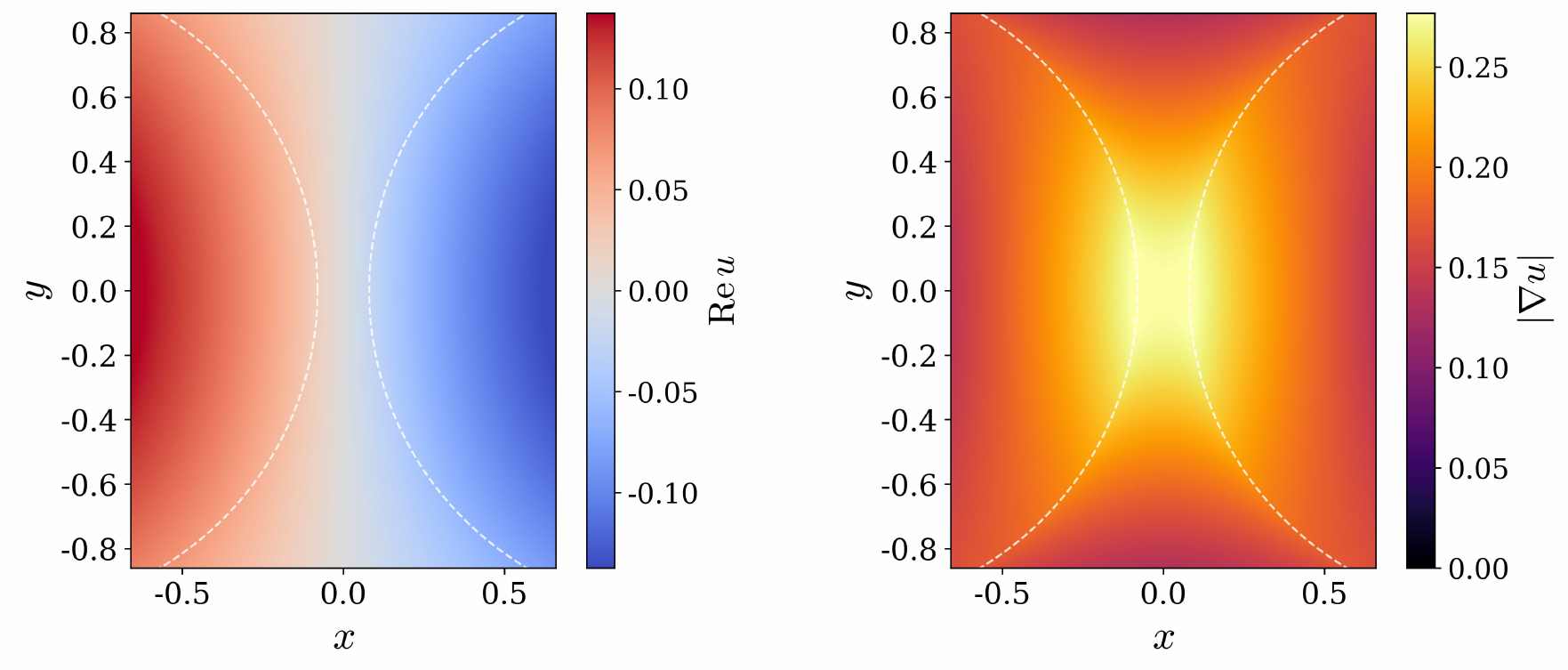}
        \subcaption{Antisymmetric combination.}
    \end{subfigure}
    \caption{Identical resonators, first radial mode $(s,m,n)=(0,0,0)$. Each panel shows the field and magnetic intensity near the gap. The antisymmetric combination produces the localized magnetic hot-spot.}
    \label{fig:identical_main_gap}
\end{figure}

Finally, Figure~\ref{fig:identical_higher_mode_gap} shows the same comparison for the higher-order mode $(s,m,n)=(4,0,1)$. Although the internal structure is more oscillatory, the same qualitative distinction remains: the antisymmetric combination leads to stronger localization in the gap than the symmetric one.

\begin{figure}[ht]
    \centering
    \begin{subfigure}[t]{0.48\linewidth}
        \centering
        \includegraphics[width=\linewidth]{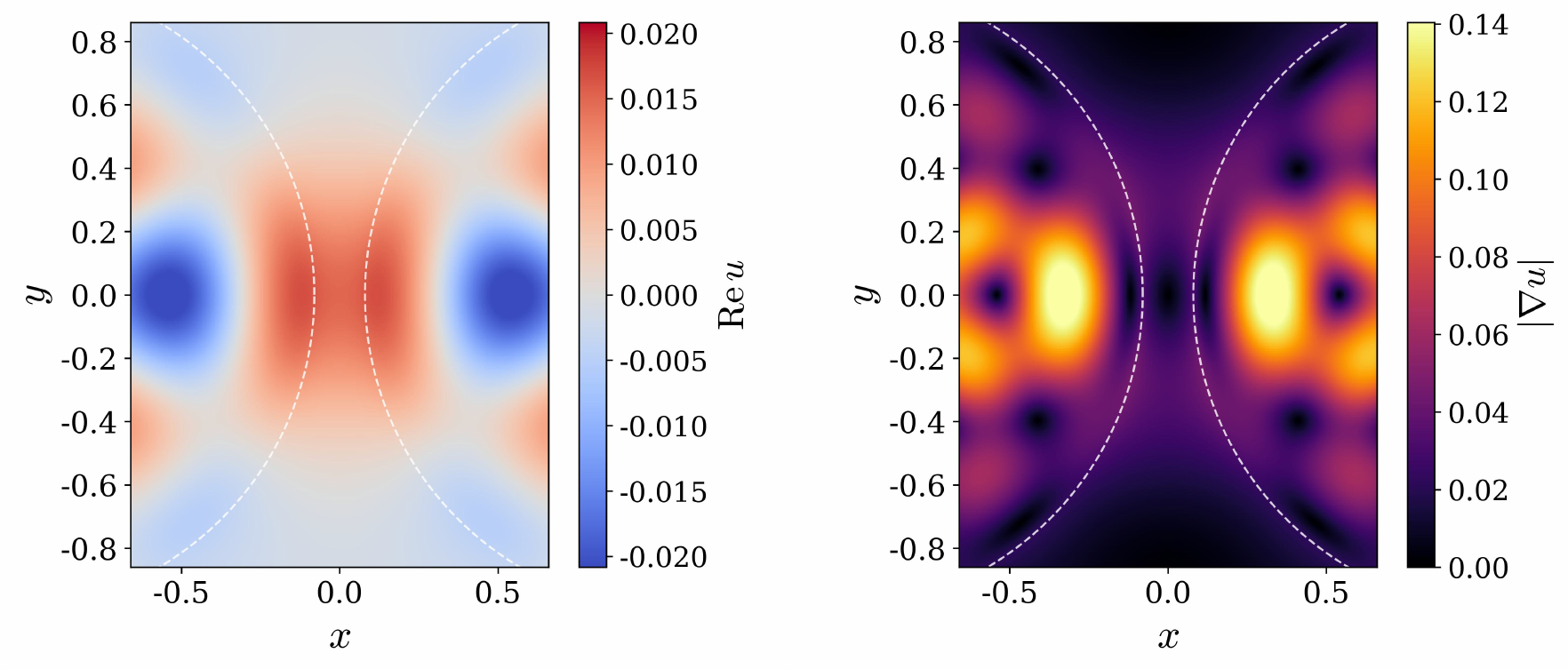}
        \subcaption{Symmetric combination.}
    \end{subfigure}
    \hfill
    \begin{subfigure}[t]{0.48\linewidth}
        \centering
        \includegraphics[width=\linewidth]{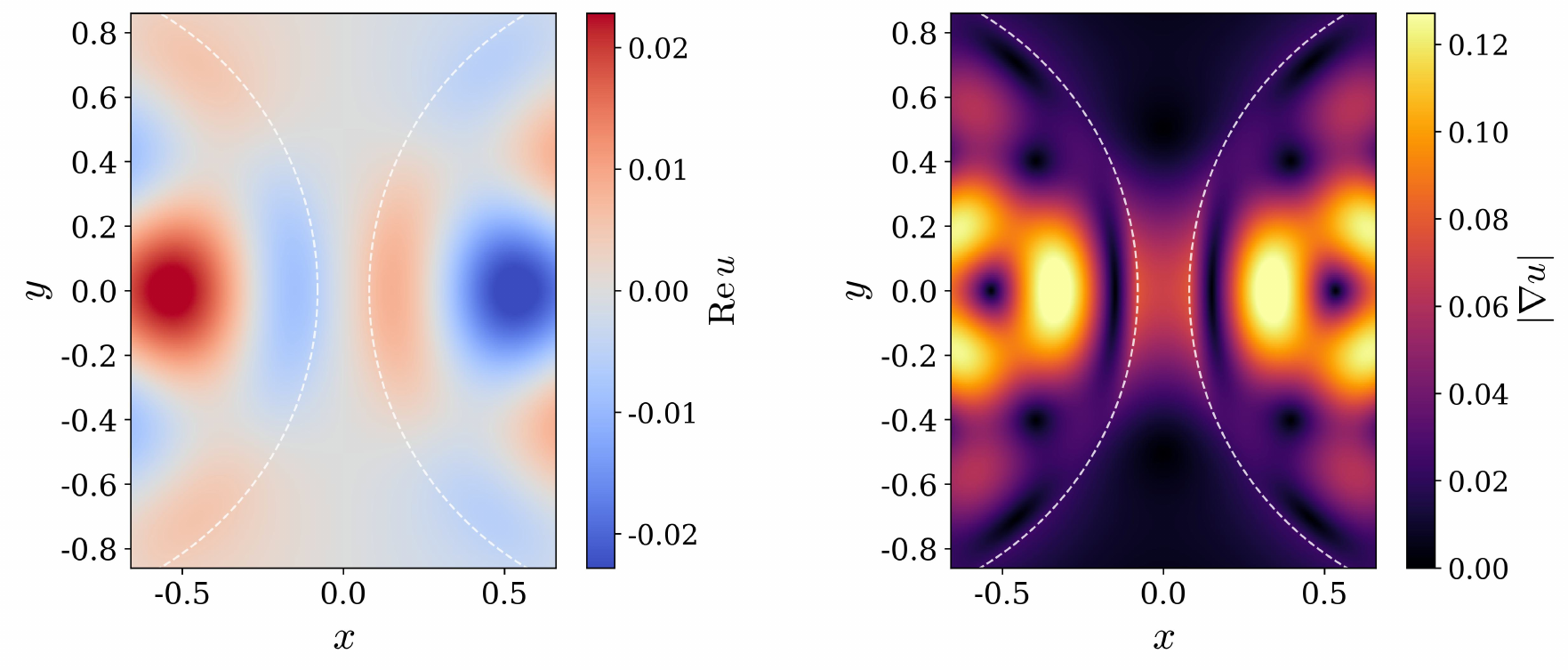}
        \subcaption{Antisymmetric combination.}
    \end{subfigure}
    \caption{Identical resonators, higher-order mode $(s,m,n)=(4,0,1)$. The antisymmetric combination again produces stronger magnetic localization in the gap.}
    \label{fig:identical_higher_mode_gap}
\end{figure}

These simulations confirm that the identical-resonator case fits the same geometric amplification picture, provided the degenerate eigenspaces are resolved into their symmetric and antisymmetric components.

\subsection{Numerical diagnostics for the mean value mechanism and saturation}

The mean value mechanism is illustrated in Figure~\ref{fig:3D_MVT_diagnostics} for the weakly coupled mode $(s,m,n)=(0,0,0)$. The left panel shows the magnitude of the gradient evaluated at the midpoint $x_{\mathrm{mid}}$ of the gap as a function of the separation distance $\kappa$. The dashed reference curve is proportional to $1/\kappa$. In the weak-coupling regime, the numerical curve follows this inverse-gap scaling, confirming that the dominant amplification mechanism is geometric: the resonant field maintains a non-negligible contrast between the two facing boundaries and is therefore forced to vary across a region of width $\kappa$.

The right panel shows the compensated quantity $\kappa|\nabla u(x_{\mathrm{mid}})|$. Over the same intermediate range of gap sizes, this quantity remains approximately constant, which is consistent with the estimate
\begin{align*}
    |\nabla u(x_{\mathrm{mid}})| \sim \frac{1}{\kappa}.
\end{align*}
For smaller gaps, deviations from this behavior indicate the onset of stronger coupling effects. In this regime, the weak-coupling modal approximation begins to break down, and the boundary contrast driving the mean value mechanism is no longer independent of $\kappa$.

\begin{figure}[ht]
    \centering
    \includegraphics[width=0.82\linewidth]{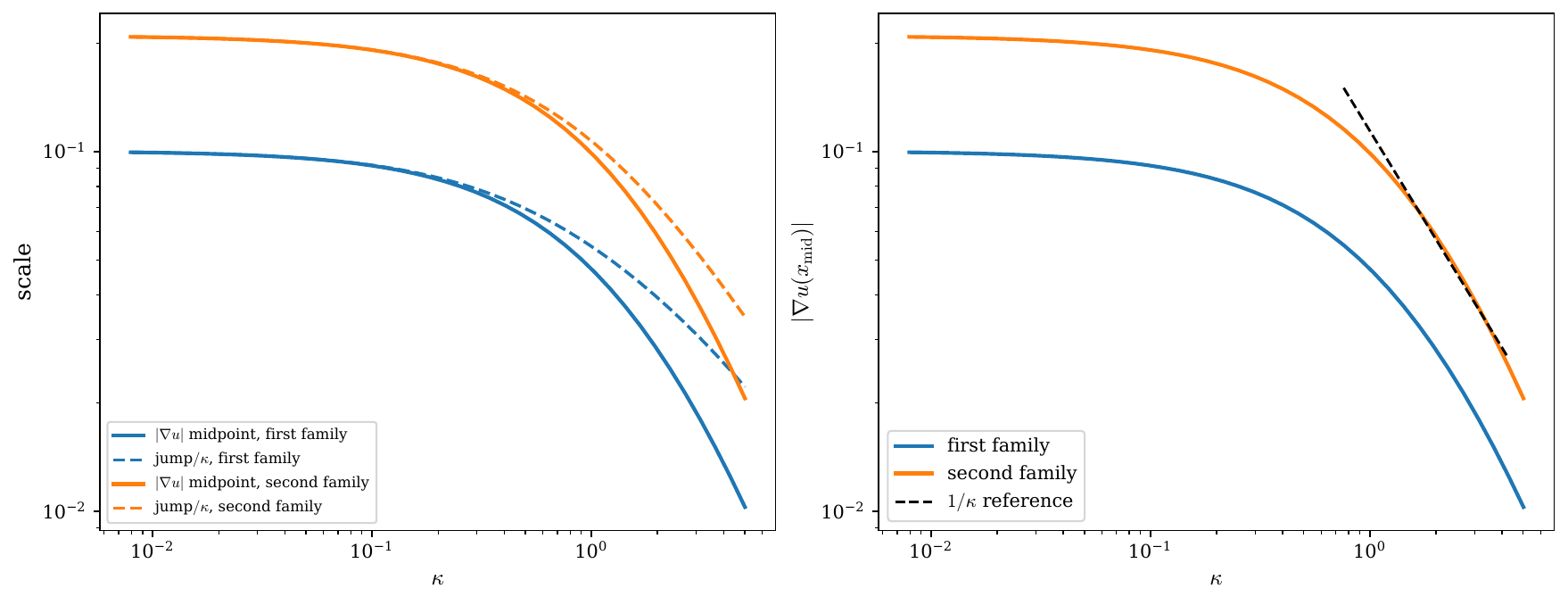}
    \caption{Mean value diagnostics for the weakly coupled mode $(s,m,n)=(0,0,0)$. The gradient at the midpoint of the gap follows the inverse-gap scaling predicted by the mean value mechanism in the intermediate weak-coupling regime. The rescaled quantity $\kappa |\nabla u(x_{\mathrm{mid}})|$ remains approximately constant over the same range, confirming the $\kappa^{-1}$ amplification.}
    \label{fig:3D_MVT_diagnostics}
\end{figure}

Figure~\ref{fig:3D_all_modes_MVT} extends this diagnostic to the three representative modes discussed in subsection~\ref{sec:modal_dependence}. The quantity $|\partial_x u(x_{\mathrm{mid}})|$, which measures the field variation in the gap direction, is plotted as a function of $\kappa$. The modes $(s,m,n)=(0,0,0)$, $(2,0,0)$ and $(4,0,1)$ all exhibit an intermediate regime that approximately follows the inverse-gap scaling predicted by the mean-value argument. The strongest response is observed for the mode $(4,0,1)$, which is also the mode producing the most pronounced hotspot in Figures~\ref{fig:mode200_hotspot}-\ref{fig:mode311_hotspot}. In contrast, the mode $(s,m,n)=(3,1,1)$ remains essentially flat and several orders of magnitude smaller throughout the entire parameter range.

These results provide a quantitative confirmation of the modal dependence of hotspot formation. Modes that generate a significant field contrast across the gap exhibit the inverse-gap amplification predicted by Section~\ref{sec:MVT}, whereas modes that do not create such a contrast fail to produce substantial gradient concentration. Moreover, all amplifying modes eventually deviate from the ideal $1/\kappa$ behavior as the gap becomes sufficiently small, reflecting the breakdown of the weak-coupling regime and the onset of the saturation mechanism described in Section~\ref{sec:MVT}.

\begin{figure}[ht]
    \centering
    \includegraphics[width=0.5\linewidth]{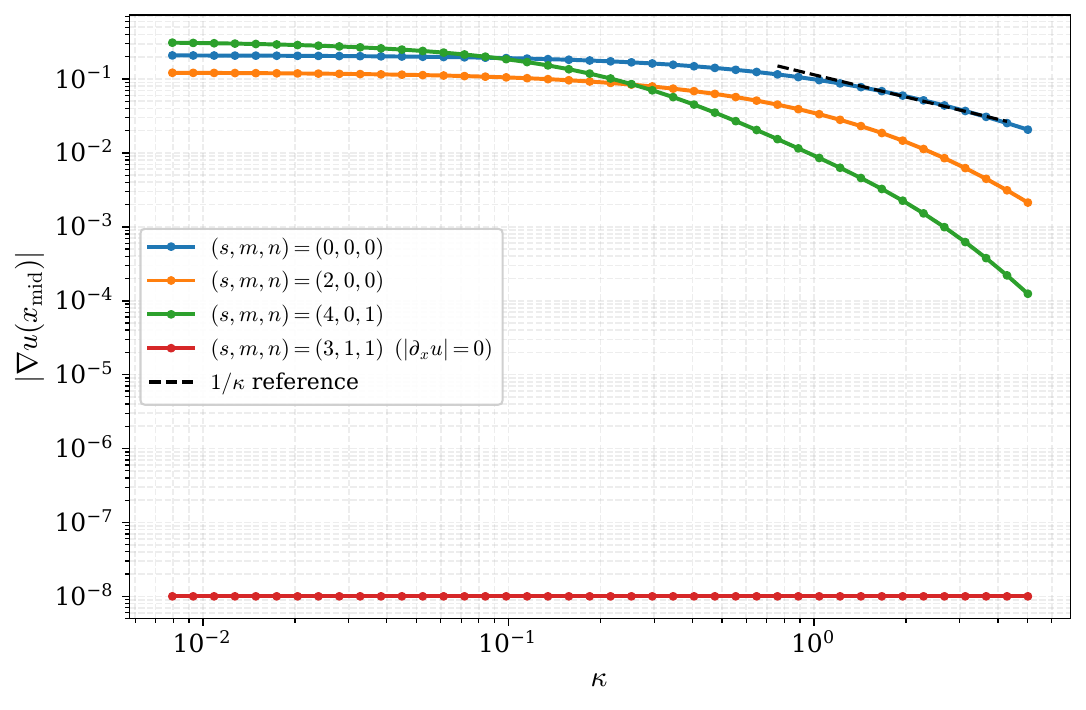}
    \caption{Comparison of $|\nabla u(x_{\mathrm{mid}})|$ for the different modes. Modes that maintain a boundary contrast exhibit inverse-gap amplification, while the mode $(s,m,n)=(3,1,1)$ remains several orders of magnitude smaller.}
    \label{fig:3D_all_modes_MVT}
\end{figure}

\subsection{Numerical investigation for spectral behavior}

The spectral diagnostic shown in Figure~\ref{fig:3D_spectral_branches} provides a complementary operator-level interpretation of the transition from weak coupling to stronger interaction. Each panel tracks one eigenvalue branch $\lambda(\kappa)$ of the coupled operator together with the corresponding isolated eigenvalue $\rho$ of the uncoupled problem. For large gap sizes, the coupled eigenvalues remain close to their isolated reference values, confirming that the spectral structure is still perturbative. As $\kappa$ decreases, the tracked branches move away from the corresponding isolated eigenvalues, showing the increasing influence of the off-diagonal interaction between the two resonators.

This departure from the isolated spectral levels is a signature of modal hybridization. In the weak-coupling regime, the eigenmodes can be interpreted as perturbative deformations of isolated resonator modes. When the gap becomes smaller, this interpretation gradually loses accuracy: the coupled eigenmodes reorganize across the two-particle system, and the isolated-mode approximation breaks down. This spectral behavior is consistent with the saturation mechanism discussed above.

\begin{figure}[ht]
    \centering
    \includegraphics[width=0.9\linewidth]{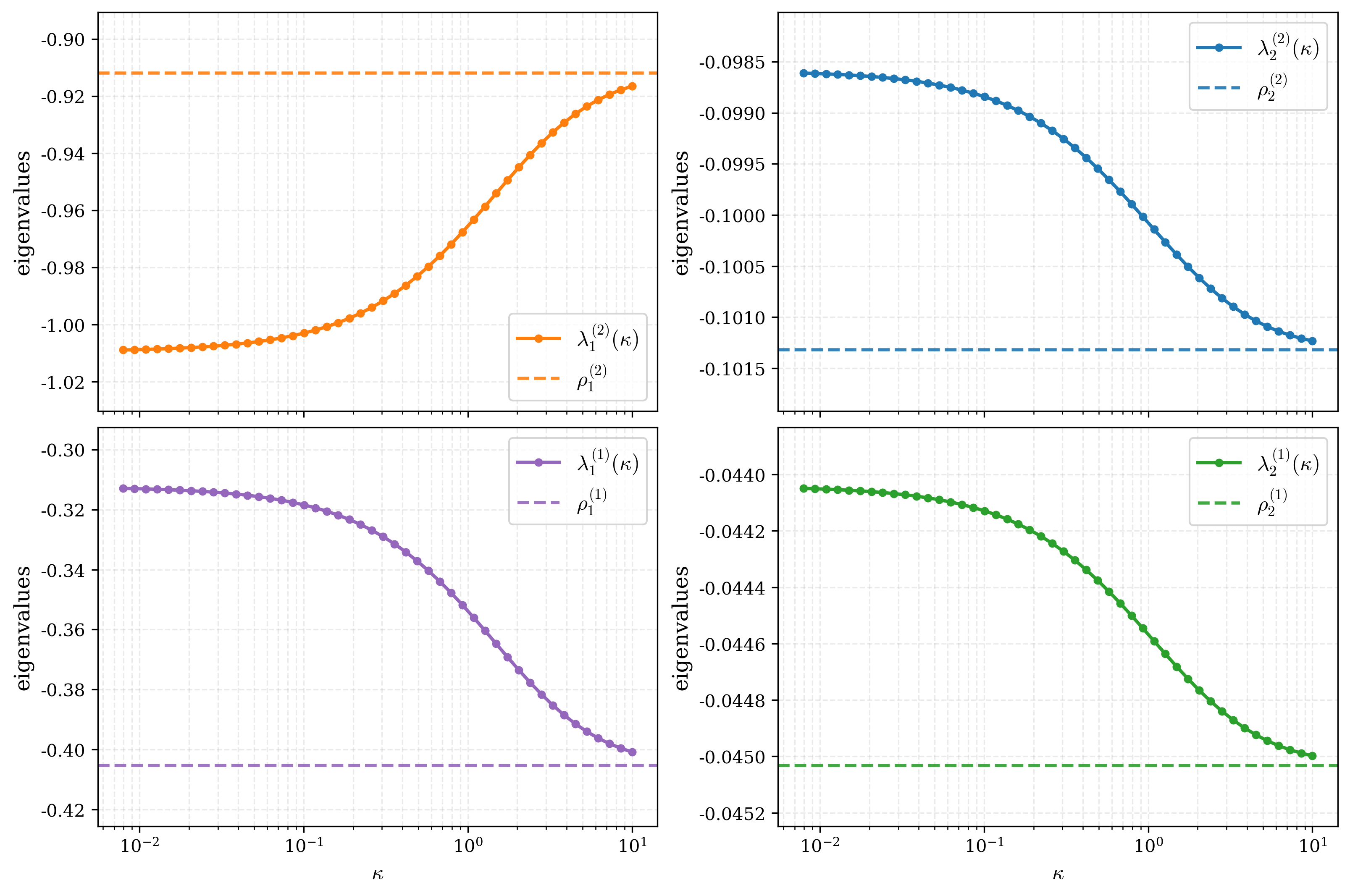}
    \caption{Tracked eigenvalue branches of the coupled operator. Solid curves show coupled eigenvalues $\lambda(\kappa)$, while dashed lines show the corresponding isolated eigenvalues $\rho$. As the gap decreases, the branches depart from their isolated references, indicating stronger inter-particle coupling and the breakdown of the perturbative isolated-mode picture.}
    \label{fig:3D_spectral_branches}
\end{figure}

\section{Conclusion}
\label{sec:Conclusion}

We have analyzed gradient hotspot formation for two weakly coupled dielectric resonators in three dimensions. Using the Lippmann--Schwinger formulation, we showed that the relevant subwavelength mechanism is governed by the Newtonian volume potential and by the interaction between the isolated spectral modes of the two resonators. In the weak-coupling regime, the coupled modes remain perturbative deformations of the isolated modes. When such a mode creates a non-vanishing contrast between the facing boundary values, the field must transition across a gap of width $\kappa$, producing an intermediate inverse-gap amplification of the gradient. This explains the observed gradient hotspots as a modal and geometric concentration effect, rather than as a genuine singularity of the underlying elliptic problem.

The amplification mechanism is limited by the breakdown of weak coupling. As the gap becomes sufficiently small, the off-diagonal interaction is no longer perturbative, the isolated-mode structure reorganizes, and the boundary contrast driving the mean value estimate is reduced, leading to saturation of the gradient. The numerical experiments confirm this transition through mean value diagnostics, weak-coupling criteria, and eigenvalue-flow plots. Natural directions for future work include extending the analysis beyond spherical resonators, treating higher-multiplicity mode families through spectral projections, and extending the scalar model to the full vector Maxwell system.

\appendix

\section{Auxiliary estimates for the pointwise control}
\label{app:Linfty-control}

\subsection{Operators and \texorpdfstring{$L^2\to C^0$}{} boundedness}

We work on
\begin{align*}
    L^2(D)=L^2(D_1)\oplus L^2(D_2),
\end{align*}
and write
\begin{align*}
    \mathcal{K}_D^{(0)}=\mathcal{K}_{\mathrm{diag}}^{(0)} + \mathcal{K}_{\mathrm{off}}^{(0)},
\end{align*}
where
\begin{align*}
    \mathcal{K}_{\mathrm{diag}}^{(0)}
    =
    \begin{pmatrix}
        \mathcal{K}_{11}^{(0)} & 0\\
        0 & \mathcal{K}_{22}^{(0)}
    \end{pmatrix},
    \qquad
    \mathcal{K}_{\mathrm{off}}^{(0)}
    =
    \begin{pmatrix}
        0 & \mathcal{K}_{12}^{(0)}\\
        \mathcal{K}_{21}^{(0)} & 0
    \end{pmatrix}.
\end{align*}
Here $\mathcal{K}_{ii}^{(0)}:L^2(D_i)\to L^2(D_i)$ denotes the self-interaction operator of the $i$-th resonator, while $\mathcal{K}_{ij}^{(0)}:L^2(D_j)\to L^2(D_i)$ denotes the interaction from $D_j$ to $D_i$.

Let
\begin{align*}
    \mathcal N_D u(x):=\frac{1}{4\pi}\int_D \frac{u(y)}{|x-y|}\,dy
\end{align*}
denote the Newtonian volume potential. With the sign convention of the present paper,
\begin{align*}
    \mathcal{K}_D^{(0)}=-\mathcal N_D.
\end{align*}
Since all estimates below are norm estimates, this sign plays no role.

We first record the elementary $L^2\to C^0$ bounds needed in the proof. For the diagonal blocks, one has
\begin{align*}
    \mathcal{K}_{ii}^{(0)}:L^2(D_i)\to C^0(\overline{D_i}),
\end{align*}
with
\begin{align*}
    \|\mathcal{K}_{ii}^{(0)}\|_{L^2(D_i)\to C^0(\overline{D_i})}
    \leq
    \frac{1}{4\pi}
    \sup_{x\in\overline{D_i}}
    \left(
        \int_{D_i}\frac{dy}{|x-y|^2}
    \right)^{1/2}
    <\infty.
\end{align*}
Similarly, if
\begin{align*}
    \kappa:=\operatorname{dist}(D_1,D_2)>0,
\end{align*}
then, for $i\neq j$,
\begin{align*}
    \|\mathcal{K}_{ij}^{(0)}\|_{L^2(D_j)\to C^0(\overline{D_i})}
    \leq
    \frac{|D_j|^{1/2}}{4\pi\kappa}.
\end{align*}

We also use the following rearrangement estimate.

\begin{lemma}[Hardy--Littlewood--Polya rearrangement inequality \cite{HardyLittlewoodPolya}]
Let $A\subset\mathbb R^3$ be measurable with finite measure, and let
$h:\mathbb R^3\to[0,\infty]$ be radial, radially non-increasing, and locally integrable. Then, for every $x\in\mathbb R^3$,
\begin{align*}
    \int_A h(x-y)\,dy
    \leq
    \int_{B_r(x)} h(x-y)\,dy,
    \qquad
    |B_r|=|A|.
\end{align*}
\end{lemma}

\begin{proof}
This is the Hardy--Littlewood--Polya rearrangement inequality applied to
$\chi_A$ and $h(x-\cdot)$.
\end{proof}

\begin{lemma}[$L^2$ to $C^0$ bound for the Newtonian potential]
Assume that $D\subset B_R$. Then
\begin{align*}
    \mathcal N_D:L^2(D)\to L^\infty(\mathbb R^3)
\end{align*}
is bounded and
\begin{align*}
    \|\mathcal N_D\|_{L^2(D)\to L^\infty(\mathbb R^3)}
    \leq
    \frac12\left(\frac{R}{\pi}\right)^{1/2}.
\end{align*}
\end{lemma}

\begin{proof}
Let $x\in\mathbb R^3$. By the Cauchy--Schwarz inequality,
\begin{align*}
    |(\mathcal N_Du)(x)|
    \leq
    \frac{1}{4\pi}
    \left(
        \int_D\frac{dy}{|x-y|^2}
    \right)^{1/2}
    \|u\|_{L^2(D)}.
\end{align*}
By the rearrangement inequality,
\begin{align*}
    \int_D\frac{dy}{|x-y|^2}
    \leq
    \int_{B_R(x)}\frac{dy}{|x-y|^2}.
\end{align*}
Moreover,
\begin{align*}
    \int_{B_R(x)}\frac{dy}{|x-y|^2}
    =
    4\pi\int_0^R dr
    =
    4\pi R.
\end{align*}
Therefore,
\begin{align*}
    |(\mathcal N_Du)(x)|
    \leq
    \frac{1}{4\pi}(4\pi R)^{1/2}\|u\|_{L^2(D)}
    =
    \frac12\left(\frac{R}{\pi}\right)^{1/2}\|u\|_{L^2(D)}.
\end{align*}
Taking the supremum over $x\in\mathbb R^3$ gives the result.
\end{proof}

\subsection{Degenerate case: identical resonators}

We finally discuss the degenerate case corresponding to identical resonators. Let $\Phi$ be a normalized simple eigenmode of one resonator, associated with the eigenvalue $\rho$. Then $\rho$ gives rise to a two-dimensional eigenspace for $\mathcal{K}_{\mathrm{diag}}^{(0)}$,
\begin{align*}
    E_\rho
    =
    \operatorname{span}
    \left\{
    (\Phi,0)^T,
    (0,\Phi)^T
    \right\}.
\end{align*}
Let $P_\rho$ denote the orthogonal projection onto $E_\rho$, and define
\begin{align*}
    d_\rho
    :=
    \operatorname{dist}
    \left(
    \rho,
    \operatorname{sp}(\mathcal{K}_{\mathrm{diag}}^{(0)})\setminus\{\rho\}
    \right)
    >0.
\end{align*}
The perturbation is diagonalized inside $E_\rho$ by the symmetric and antisymmetric modes
\begin{align*}
    u_+
    :=
    \frac{1}{\sqrt2}(\Phi,\Phi)^T,
    \qquad
    u_-
    :=
    \frac{1}{\sqrt2}(\Phi,-\Phi)^T.
\end{align*}
Indeed, in the basis
\begin{align*}
    \left((\Phi,0)^T,(0,\Phi)^T\right),
\end{align*}
the restriction $P_\rho \mathcal{K}_{\mathrm{off}}^{(0)}P_\rho$ has the form
\begin{align*}
    \begin{pmatrix}
        0 & \alpha\\
        \alpha & 0
    \end{pmatrix},
\end{align*}
where
\begin{align*}
    \alpha
    :=
    \left\langle
    \mathcal{K}_{\mathrm{off}}^{(0)}(\Phi,0)^T,
    (0,\Phi)^T
    \right\rangle_{L^2}
    =
    \left\langle
    \mathcal{K}_{21}^{(0)}\Phi,\Phi
    \right\rangle_{L^2(D_2)}.
\end{align*}
Therefore the first-order shifts are
\begin{align*}
    \mu_\pm=\pm\alpha,
\end{align*}
with eigenvectors $u_\pm$.

\begin{remark}
The main difference with the non-degenerate case is that the diagonal matrix element of the perturbation inside the relevant eigenspace does not vanish:
\begin{align*}
    \langle \mathcal{K}_{\mathrm{off}}^{(0)}u_\pm,u_\pm\rangle_{L^2}
    =
    \pm\alpha.
\end{align*}
Hence the eigenvalue shift has a non-zero first-order contribution of size
$|\alpha|\leq \varepsilon_2$, and the sharper bound
\begin{align*}
    |\delta\rho|=O\left(\frac{\varepsilon_2^2}{d_\rho}\right)
\end{align*}
is not available in general.
\end{remark}

The perturbed eigenpairs splitting from $\rho$ satisfy
\begin{align*}
    \rho_{\kappa,\pm}
    =
    \rho+\Delta\rho_\pm,
\end{align*}
where
\begin{align*}
    \Delta\rho_\pm
    =
    \pm\alpha+\delta\rho_\pm^{(2)},
    \qquad
    |\delta\rho_\pm^{(2)}|
    =
    O\left(
        \frac{\varepsilon_2^2}{d_\rho}
    \right).
\end{align*}
Moreover,
\begin{align*}
    \|\Phi_{\kappa,\pm}-u_\pm\|_{L^2}
    =
    O\left(
        \frac{\varepsilon_2}{d_\rho}
    \right).
\end{align*}
Here $u_\pm$ is an exact eigenvector of $\mathcal{K}_{\mathrm{diag}}^{(0)}$ with eigenvalue $\rho$, since
\begin{align*}
    \mathcal{K}_{\mathrm{diag}}^{(0)}u_\pm
    =
    \frac{1}{\sqrt2}
    \left(
    \mathcal{K}_{11}^{(0)}\Phi,
    \pm \mathcal{K}_{22}^{(0)}\Phi
    \right)^T
    =
    \rho u_\pm.
\end{align*}

We can now repeat the direct $C^0$ argument.

\begin{proposition}[Direct $C^0$ estimate in the degenerate case]
\label{prop:directC0_deg}
Set $\delta\phi_\pm := \Phi_{\kappa,\pm}-u_\pm$ and assume that
\begin{align*}
    \frac{|\Delta\rho_\pm|}{|\rho|}
    \leq
    \frac12.
\end{align*}
Then
\begin{align*}
    \|\delta\phi_\pm\|_{C^0}
    \leq
    \frac{2}{|\rho|}
    \left(
    C_0\|\delta\phi_\pm\|_{L^2}
    +
    \varepsilon_\infty
    +
    \varepsilon_\infty\|\delta\phi_\pm\|_{L^2}
    +
    |\Delta\rho_\pm|\|u_\pm\|_{C^0}
    \right).
\end{align*}
\end{proposition}

\begin{proof}
From
\begin{align*}
    (\mathcal{K}_{\mathrm{diag}}^{(0)}+\mathcal{K}_{\mathrm{off}}^{(0)})\Phi_{\kappa,\pm}
    =
    (\rho+\Delta\rho_\pm)\Phi_{\kappa,\pm}
\end{align*}
and
\begin{align*}
    \mathcal{K}_{\mathrm{diag}}^{(0)}u_\pm=\rho u_\pm,
\end{align*}
subtracting gives
\begin{align*}
    \rho\,\delta\phi_\pm
    =
    \mathcal{K}_{\mathrm{diag}}^{(0)}\delta\phi_\pm
    +
    \mathcal{K}_{\mathrm{off}}^{(0)}u_\pm
    +
    \mathcal{K}_{\mathrm{off}}^{(0)}\delta\phi_\pm
    -
    \Delta\rho_\pm u_\pm
    -
    \Delta\rho_\pm\delta\phi_\pm.
\end{align*}
Taking the $C^0(\overline D)$ norm and using
\begin{align*}
    \|\mathcal{K}_{\mathrm{diag}}^{(0)}v\|_{C^0}
    \leq
    C_0\|v\|_{L^2},
    \qquad
    \|\mathcal{K}_{\mathrm{off}}^{(0)}v\|_{C^0}
    \leq
    \varepsilon_\infty\|v\|_{L^2},
\end{align*}
together with $\|u_\pm\|_{L^2}=1$, gives
\begin{align*}
    \|\delta\phi_\pm\|_{C^0}
    \leq
    \frac{1}{|\rho|}
    \left(
    C_0\|\delta\phi_\pm\|_{L^2}
    +
    \varepsilon_\infty
    +
    \varepsilon_\infty\|\delta\phi_\pm\|_{L^2}
    +
    |\Delta\rho_\pm|\|u_\pm\|_{C^0}
    +
    |\Delta\rho_\pm|\|\delta\phi_\pm\|_{C^0}
    \right).
\end{align*}
The assumption $|\Delta\rho_\pm|/|\rho|\leq 1/2$ allows the last term to be absorbed into the left-hand side. This proves the estimate.
\end{proof}

Using
\begin{align*}
    \|\delta\phi_\pm\|_{L^2}
    \leq
    C\frac{\varepsilon_2}{d_\rho},
\end{align*}
the identity
\begin{align*}
    u_\pm = \rho^{-1} \mathcal{K}_{\mathrm{diag}}^{(0)} u_\pm,
\end{align*}
which implies
\begin{align*}
    \|u_\pm\|_{C^0}
    \leq
    \frac{C_0}{|\rho|},
\end{align*}
and the eigenvalue estimate
\begin{align*}
    |\Delta\rho_\pm|
    \leq
    C\varepsilon_2,
\end{align*}
we obtain the following.

\begin{corollary}[Pointwise estimate in the degenerate case]
\label{cor:directC0_deg}
Under the weak-coupling assumptions and
\begin{align*}
    \frac{|\Delta\rho_\pm|}{|\rho|}
    \leq
    \frac12,
\end{align*}
one has
\begin{align*}
    \|\delta\phi_\pm\|_{C^0}
    \leq
    C\left(
    \frac{C_0}{|\rho|}\frac{\varepsilon_2}{d_\rho}
    +
    \frac{\varepsilon_\infty}{|\rho|}
    +
    \frac{\varepsilon_\infty}{|\rho|}\frac{\varepsilon_2}{d_\rho}
    +
    \frac{C_0}{|\rho|^2}\varepsilon_2
    \right),
\end{align*}
where $C$ is independent of $\varepsilon_2$ and $\varepsilon_\infty$.
\end{corollary}

\bibliographystyle{abbrv}
\bibliography{references}{}

\end{document}